\documentclass[11pt,letterpaper]{article}

\usepackage{fullpage}

\usepackage{amssymb,amsmath}
\usepackage{amsthm}
\usepackage{graphicx}
\usepackage{color}

\usepackage{epstopdf}

\usepackage{algorithm}
\usepackage{algorithmic}

\newtheorem{thm}{Theorem}
\newtheorem{lem}{Lemma}
\newtheorem{ass}{Assumption}

\newtheorem{prop}{Proposition}

\theoremstyle{definition}

\newtheorem{example}{Example}

\newcommand{\abs}[1]{\lvert #1 \rvert}
\newcommand{\norm}[1]{\lVert #1 \rVert}

\newcommand{\floor}[1]{\lfloor #1 \rfloor}

\newcommand{\ve}[2]{\langle #1 ,  #2 \rangle}   
\newcommand{\eqdef}{\stackrel{\text{def}}{=}}
\newcommand{\R}{\mathbf{R}}
\newcommand{\Reg}{\psi}

\newcommand{\pp}{u}
\newcommand{\vv}{v}
\newcommand{\Exp}{\mathbf{E}}
\newcommand{\Prob}{\mathbf{P}}

\newcommand{\Srv}{\hat{S}}

\DeclareMathOperator{\dom}{dom}

\title{Accelerated, Parallel and Proximal Coordinate Descent}

\author{Olivier Fercoq \footnote{School of Mathematics, The University of Edinburgh, United Kingdom (e-mail: olivier.fercoq@ed.ac.uk)} \qquad \qquad  Peter Richt\'{a}rik \footnote{School of Mathematics, The University of Edinburgh, United Kingdom (e-mail: peter.richtarik@ed.ac.uk) \qquad
The work of both authors was supported by the EPSRC grant EP/I017127/1 (Mathematics for Vast Digital Resources) and 
by the Centre for Numerical Algorithms and Intelligent Software (funded by EPSRC grant EP/G036136/1 and the Scottish Funding Council). The work of P.R.\ was also supported by EPSRC grant EP/K02325X/1 (Accelerated Coordinate Descent Methods for Big Data Problems) and by the Simons Institute for the Theory of Computing at UC Berkeley.}}

\date{December 19, 2013 (updated: February 2014)}

\begin{document}

\maketitle

\begin{abstract} We propose a new stochastic coordinate descent method for minimizing the sum of convex functions each of which depends on a small number of coordinates only. Our method (APPROX) is simultaneously Accelerated, Parallel and PROXimal; this is the first time such a method is proposed. In the special case when the number of processors is equal to the number of coordinates, the method converges at the rate $2\bar{\omega}\bar{L} R^2/(k+1)^2 $, where $k$ is the iteration counter, $\bar{\omega}$ is an \emph{average} degree of separability of the loss function, $\bar{L}$ is the \emph{average} of  Lipschitz constants associated with the coordinates and  individual functions in the sum, and $R$ is the distance of the initial point from the minimizer. We show that the  method can be implemented without the need to perform full-dimensional vector operations, which is  the major bottleneck of accelerated coordinate descent.  The fact that the method depends on the average degree of separability, and not on the maximum degree of separability, can be attributed to the use of new safe large stepsizes, leading to improved expected separable overapproximation (ESO). These are of independent interest and can be utilized in all existing parallel stochastic coordinate descent algorithms based on the concept of ESO. 
\end{abstract}

\section{Introduction}

Developments in computing technology and  ubiquity of digital devices resulted in an increased interest in solving optimization problems of extremely big sizes. Applications can be found in all areas of human endeavor where data is available, including the internet, machine learning, data science and scientific computing. The size of these problems is so large that it is necessary to decompose the problem into smaller, more manageable, pieces.  Traditional approaches, where it is possible to rely on full-vector operations in the design of an iterative scheme, must be revisited.

Coordinate descent methods \cite{Nesterov:2010RCDM, RT:UCDC} appear as a very popular class of algorithms for such problems as they can break down the problem into smaller pieces, and can take advantage of sparsity patterns in the data. With big data problems it is necessary to design algorithms able to utilize modern parallel computing architectures. This resulted in an interest
in  parallel \cite{RT:PCDM, takac2013mini,FR:spcdm,RT:2013optimal} and distributed \cite{RT:2013distributed} coordinate descent methods.

In this work we focus on the solution of convex optimization problems with a huge number of variables of the form
\begin{equation}\label{eq:main-intro}\min_{x\in \R^N}\;\; f(x) + \Reg(x).\end{equation}
Here $x = (x^{(1)},\dots,x^{(n)}) \in \R^N$ is a decision vector composed of $n$ blocks, with $x^{(i)}\in \R^{N_i}$, \begin{equation}\label{eq:shs5s9s}\textstyle{f(x)=\sum_{j=1}^m f_j(x)},\end{equation} where $f_j$ are smooth convex functions, and $\Reg$ is a block separable regularizer (e.g., $L1$ norm).

In this work we make the following three main contributions:

\begin{enumerate}
\item We design and analyze the first stochastic coordinate descent method which is simultaneously \emph{accelerated}, \emph{parallel} and \emph{proximal}. In fact,  we are not aware of any published results on accelerated coordinate descent which would either be proximal \emph{or} parallel. 

 Our method is \emph{accelerated} in the sense that it achieves an $O(1/k^2)$ convergence rate, where $k$ is the iteration counter. The first \emph{gradient} method with this convergence rate is due to Nesterov~\cite{nesterov1983method}; see also \cite{tseng2008accelerated, beck2009fista}.
Accelerated stochastic coordinate descent method, for convex minimization without constraints, was originally proposed in 2010 by Nesterov~\cite{Nesterov:2010RCDM}.

\begin{table}[!h]
\begin{center}
{
\footnotesize
\begin{tabular}{|l||c|c|c||c|}
\hline
{\bf Paper} & {\bf Proximal} &{\bf Parallel} & {\bf Accelerated } &
{\bf  Notable feature} \\
 \hline
Leventhal \& Lewis, 2008 
\cite{Leventhal:2008:RMLC} 
& $\times$ & $\times$ & $\times$ & quadratic $f$\\
S-Shwartz \& Tewari, 2009 
\cite{ShalevTewari09} 
& $\ell_1$ & $\times$ & $\times$ & 1st $\ell_1$-regularized\\
Nesterov,  2010 
\cite{Nesterov:2010RCDM} 
& $\times$  & $\times$ & YES & 1st block, 1st accelerated \\
Richt\'{a}rik \& Tak\'{a}\v{c},  2011 
\cite{RT:UCDC} 
& YES & $\times$ & $\times$ & 1st general proximal \\
Bradley et al, 2012 
\cite{Bradley:PCD-paper} 
& $\ell_1$ & YES & $\times$ & $\ell_1$-regularized parallel\\
Richt\'{a}rik \& Tak\'{a}\v{c}, 2012 
\cite{RT:PCDM}  
& YES & YES & $\times$ & 1st general parallel\\
S-Shwartz \& Zhang, 2012 
\cite{Proximal-dual-Coord-Ascent} 
& YES & $\times$ & $\times$ &  1st primal-dual\\
Necoara et al, 2012 \cite{Necoara:Coupled} & $\times$ & $\times$ & $\times$ & 2-coordinate descent \\
Tak\'{a}\v{c} et al, 2013 
\cite{takac2013mini} 
& $\times$ & YES & $\times$ & 1st primal-d. \& parallel\\
Tappenden et al, 2013 
\cite{GRT:Inexact} 
& YES & $\times$ & $\times$ & 1st inexact\\
Necoara \& Clipici, 2013 \cite{Necoara:parallelCDM-MPC} & YES & $\times $ & $\times$ & coupled constraints\\
Lin \& Xiao, 2013 
\cite{LX:2013-improvedCDM}
& $\times$ & $\times$ & YES & improvements\\
Fercoq \& Richt\'{a}rik,  2013 
\cite{FR:spcdm} 
& YES  & YES & $\times$ & 1st nonsmooth $f$\\
Lee \& Sidford, 2013 
\cite{lee2013efficient} 
& $\times$ & $\times$ & YES &  1st efficient accelerated \\
Richt\'{a}rik \& Tak\'{a}\v{c}, 2013 \cite{RT:2013distributed} 
& YES & YES & $\times$ & 1st distributed\\
Liu et al, 2013 \cite{Wright:2013-async_par_cdm}  & $\times$ & YES & $\times$ & asynchronous\\
Richt\'{a}rik \& Tak\'{a}\v{c}, 2013 
\cite{RT:2013optimal} 
& $\times$ & YES  & $\times$ &  1st parallel  nonuniform \\
\hline
\textbf{This paper} & {\bf YES} & {\bf YES} & {\bf YES} & {\bf 3 $\times$ YES }\\
\hline
 \end{tabular}
 }
 \end{center}
 \caption{\footnotesize Selected recent papers  analyzing the iteration complexity of \emph{stochastic} coordinate descent methods. Our algorithm is simultaneously proximal, parallel and accelerated. In the last column we highlight a single notable feature, necessarily chosen subjectively, of each work.   }
 \label{tbl:3xYES1}
\end{table}

Various variants of proximal and parallel (but non-accelerated) stochastic coordinate descent methods were proposed \cite{Bradley:PCD-paper,RT:PCDM,FR:spcdm,RT:2013distributed}. In Table~\ref{tbl:3xYES1} we provide a  list\footnote{This list is necessarily incomplete, it was not our goal to be comprehensive. For a somewhat more substantial review of these and other works we refer the reader to \cite{RT:PCDM,FR:spcdm}.} of some recent research papers proposing and analyzing  \emph{stochastic} coordinate descent methods.  The table substantiates our observation that while the proximal setting is standard in the literature, parallel methods are much less studied, and finally, there  is just a handful of papers dealing with accelerated variants. 

\item We propose \emph{new  stepsizes} for parallel coordinate descent methods, based on a new expected separable overapproximation (ESO). These stepsizes can for some classes of problems (e.g., $f_j$=quadratics), be much larger than the stepsizes proposed for the (non-accelerated) parallel coordinate descent method (PCDM) in \cite{RT:PCDM}. Let $\omega_j$ be the number of of blocks function $f_j$ depends on. The stepsizes, and hence the resulting complexity, of PCDM, depend on the quantity $\omega=\max_j \omega_j$. However, our stepsizes take all the values $\omega_j$ into consideration and the result of this is complexity that depends on a data-weighted average $\bar{\omega}$ of the values $\omega_j$. Since $\bar{\omega}$ can be much smaller than $\omega$, our stepsizes result in dramatic acceleration for our method and other methods whose analysis is based on an ESO \cite{RT:PCDM,FR:spcdm, RT:2013distributed}.

\item We identify a large subclass of  problems of the form \eqref{eq:main-intro} for which the \emph{full-vector operations} inherent in accelerated methods \emph{can be eliminated.} This contrasts with Nesterov's accelerated coordinate descent scheme \cite{Nesterov:2010RCDM}, which is impractical due to this bottleneck. Having established his convergence result, Nesterov remarked \cite{Nesterov:2010RCDM}  that:

\begin{quote}
{\footnotesize
``However, for some applications [...] the complexity of one iteration of the accelerated scheme is rather high since for computing $y_k$ it needs to operate with full-dimensional vectors.''
}
\end{quote} 

Subsequently, in part due to these issues, the work of the community focused on simple methods as opposed to accelerated variants. For instance, Richt\'{a}rik \& Tak\'{a}\v{c}   \cite{RT:UCDC} use Nesterov's observation to justify their focus  on non-accelerated methods in their work on coordinate descent methods in the proximal/composite setting.

Recently, Lee \& Sidford \cite{lee2013efficient} were able to avoid full dimensional operations in the case of minimizing a convex quadratic without constraints, by a careful modification of Nesterov's method. This was achieved by introducing  an extra sequence of iterates and  observing that for quadratic functions it is possible to compute partial derivative of $f$ evaluated at a linear combination of full dimensional vectors without ever forming the combination. 
We  extend the ideas of Lee \& Sidford \cite{lee2013efficient} to our general setting \eqref{eq:main-intro} in the case when  $f_j(x) = \phi_j(a_j^Tx)$, where $\phi_j$ are scalar convex functions with Lipschitz derivative and the vectors $a_j$ are block-sparse.

 
\end{enumerate}




\paragraph{Contents.}

The rest of the paper is organized as follows. We start by describing  new stepsizes for parallel coordinate descent methods,  based on novel assumptions, and compare them with existing stepsizes (Section~\ref{SEC:STEPSIZES}). We then describe our algorithm and state and comment on the main complexity result (Section~\ref{SEC:ACD}). Subsequently,  we give a proof of the result (Section~\ref{SEC:COMPLEXITY}). We then describe an efficient implementation of our method, one that does not require the computation of full-vector operations (Section~\ref{SEC:EFFICIENT_IMPLEMENTATION}), and finally comment on our numerical experiments (Section \ref{SEC:NUMERICS}).

\paragraph{Notation.} It will be convenient to define natural operators acting between the spaces $\R^N$ and $\R^{N_i}$. In particular,  we will often wish to lift a block $x^{(i)}$ from $\R^{N_i}$ to $\R^N$, filling the coordinates corresponding to the remaining blocks with zeros. Likewise, we will project $x\in \R^{N}$ back into $\R^{N_i}$. We will now formalize these operations.

Let $U$ be the $N\times N$ identity matrix, and let $U=[U_1,U_2,\dots,U_n]$ be its decomposition into  column submatrices $U_i \in \R^{N \times N_i}$. For $x\in \R^N$, let $x^{(i)}$ be the block of variables corresponding to the columns of $U_i$, that is, $x^{(i)} = U_i^T x \in \R^{N_i}$, $i=1,2,\dots,n$. Any vector $x \in \R^N$ can be written, uniquely, as $x = \sum_{i=1}^n U_i x^{(i)}$. For $h \in \R^N$ and $\emptyset \neq S\subseteq [n]\eqdef \{1,2,\dots,n\}$, we write
\begin{equation}
\label{eq:h_S}h_{[S]} = \sum_{i\in S} U_i h^{(i)}.\end{equation}
In words, $h_{[S]}$ is a vector in $\R^N$ obtained from $h\in \R^N$ by zeroing out the blocks that do not belong to $S$. For convenience, we will also write
\begin{equation}\nabla_i f(x) \eqdef (\nabla f(x))^{(i)} = U_i^T \nabla f(x) \in \R^{N_i} \label{eq:syshsj9s}\end{equation}
for the vector of partial derivatives of $f$ corresponding to coordinates belonging to block $i$.

With each block $i\in[n]$ we associate a positive definite matrix $B_i \in \R^{N_i \times N_i}$ and a scalar $\vv_i>0$, and equip $\R^{N_i}$ and $\R^N$ with the norms
\begin{equation}\label{eq:norm_block} \textstyle \|x^{(i)}\|_{(i)}\eqdef\ve{B_i x^{(i)}}{x^{(i)}}^{1/2}, \qquad \|x\|_{\vv}  \eqdef \left(\sum_{i=1}^n \vv_i \|x^{(i)}\|_{(i)}^2\right)^{1/2}.
\end{equation}
The corresponding conjugate norms, defined by $\|s\|^*=\max\{\ve{s}{x}:\|x\|\leq 1\}$, are given by
\begin{equation}\label{eq:norm_block_conj} \textstyle \|x^{(i)}\|^*_{(i)}\eqdef\ve{B_i^{-1} x^{(i)}}{x^{(i)}}^{1/2}, \qquad \|x\|^*_{\vv}  = \left(\sum_{i=1}^n \vv_i^{-1} \left(\|x^{(i)}\|_{(i)}^*\right)^2\right)^{1/2}.
\end{equation}
We also write $\|\vv\|_1 = \sum_i |\vv_i|$.

\section{Stepsizes for parallel coordinate descent methods}\label{SEC:STEPSIZES}

The  framework for designing and analyzing (non-accelerated) parallel coordinate descent methods, developed by Richt\'{a}rik \& Tak\'{a}\v{c} \cite{RT:PCDM}, is  based on the notions of  \emph{block sampling}  and   \emph{expected separable overapproximation (ESO)}. We now briefly review this framework  as our accelerated method is cast in it, too. Informally, a block sampling is the random law describing the \emph{selection of blocks}  at each iteration. An ESO is an inequality, involving $f$ and $\hat{S}$, which is used to \emph{compute updates} to selected blocks. The complexity analysis in our paper is based on the following generic assumption.


\begin{ass}[Expected Separable Overapproximation \cite{RT:PCDM,FR:spcdm}] \label{ass:ESO} 
We assume that:

\begin{enumerate}
\item $f$ is convex and differentiable. 

\item $\hat{S}$ is a uniform block sampling. That is, $\hat{S}$ is a random subset of $[n]= \{1,2,\dots,n\}$ with the property\footnote{It is easy to see that if $\hat{S}$ is a uniform sampling, then  necessarily, $\Prob(i\in \hat{S}) = \frac{\Exp|\hat{S}|}{n}$ for all $i \in [n]$.
} that $\Prob(i \in \hat{S})=\Prob(j\in \hat{S})$ for all $i,j \in [n]$. Let $\tau=\Exp[|\hat{S}|]$. 

\item There are computable constants  $\vv=(\vv_1,\dots,\vv_n)>0$ for which the pair $(f,\hat{S})$ admits the Expected Separable Overapproximation (ESO):
\begin{equation}\label{eq:ESO}\Exp\left[f(x+h_{[\hat{S}]})\right] \leq f(x) + \frac{\tau}{n}\left(\ve{\nabla f(x)}{h} + \frac{1}{2}\|h\|_{\vv}^2\right), \qquad x,h \in \R^N.
\end{equation}
\end{enumerate}
If the above inequality holds, for simplicity we will write\footnote{In \cite{RT:PCDM}, the authors write $\tfrac{\beta}{2}\|h\|^2_w$  instead of $\tfrac{1}{2}\|h\|_{\vv}^2$. This is because 
they study families of samplings $\hat{S}$, parameterized by $\tau$, for which $w$ is fixed and all changes can thus be captured in the constant $\beta$. Clearly, the two definitions are interchangeable as one can choose $\vv=\beta w$. Here we will need to compare weights which are not linearly dependent, hence the simplified notation.} $(f,\hat{S})\sim \mathrm{ESO}(\vv)$. 
\end{ass}

In the context of parallel coordinate descent methods, uniform  block samplings and inequalities \eqref{eq:ESO} involving such samplings were introduced and systematically studied by Richt\'{a}rik \& Tak\'{a}\v{c} \cite{RT:PCDM}. An ESO inequality for a uniform \emph{distributed} sampling was developed in \cite{RT:2013distributed} and that \emph{nonuniform} samplings and ESO, together with a parallel coordinate descent method based on such samplings, was proposed in \cite{RT:2013optimal}.

Fercoq \& Richt\'{a}rik  \cite[Theorem 10]{FR:spcdm} observed that inequality \eqref{eq:ESO} is equivalent to requiring that the gradients  of the functions  \[\hat{f}_x: h \mapsto  \Exp\left[f(x+h_{[\hat{S}]})\right], \qquad x \in \R^N,\]  be Lipschitz at $h=0$, uniformly in $x$, with constant $\tau/n$, with respect to the norm $\|\cdot\|_{\vv}$. Equivalently, the Lipschitz constant is $L^{\hat{f}}$ with respect to the norm $\|\cdot\|_{\tilde{\vv}}$, where
\[L^{\hat{f}}=\frac{\tau\| \vv \|_1}{n^2}, \qquad \tilde{\vv} \eqdef n \frac{\vv}{\|\vv\|_1}.\]
The change of norms is done so as to enforce that the weights in the norm sum to $n$, which means that different ESOs can be compared using the constants $L^{\hat{f}}$.
The above observations are useful in understanding what the ESO inequality encodes: By moving from $x$ to \[x_+ = x+h_{[\hat{S}]},\] one is taking a step in a random subspace of $\R^N$ spanned by the blocks belonging to $\hat{S}$. If $\tau\ll n$, which is often the case in big data problems\footnote{In fact, one may define a ``big data'' problem by requiring that the number of parallel processors $\tau$ available for optimization is much smaller than the dimension $n$ of the problem.}, the step is confined to a \emph{low-dimensional} subspace of $\R^N$. It turns out that for many classes of functions arising in applications, for instance for functions exhibiting certain sparsity or partial separability patterns,  it is the case that the gradient of $f$ varies much more slowly in such subspaces, on average,  than it does in $\R^N$. This in turn would imply that updates $h$ based on minimizing the right hand side of \eqref{eq:ESO} would produce larger steps, and eventually lead to faster convergence. 



\subsection{New model}

Consider $f$  of the form \eqref{eq:shs5s9s}, i.e., 
\[f(x) =\sum_{j=1}^m f_j(x),\]
where $f_j$ depends on blocks $i \in C_j$ only. Let $\omega_j=|C_j|$, and $\omega=\max_j \omega_j$.

%
%
%

\begin{ass}\label{ASS:4}
The functions $\{f_j\}$  have  block-Lipschitz gradient with constants $L_{ji}\geq 0$. That is, for all $j=1,2,\dots,m$ and $i=1,2,\dots,n$,  
\begin{equation}\label{eq:newass}\|\nabla_i f_j(x+ U_i t)-\nabla_i f_j(x)\|_{(i)}^* \leq L_{ji} \|t\|_{(i)}, \qquad x \in \R^N, \; t \in \R^{N_i}.
\end{equation}
\end{ass}

Note that, necessarily, \begin{equation}L_{ji} = 0 \qquad  \text{whenever} \qquad i \notin C_j.\end{equation}

Assumption~\ref{ASS:4} is \emph{stronger} than the assumption considered in \cite{RT:PCDM}. Indeed, in \cite{RT:PCDM} the authors only assumed that the \emph{sum} $f$,  as opposed to the individual functions $f_j$,  has a block-Lipschitz gradient, with constants $L_1,\dots, L_n$. That is,

\[\|\nabla_i f(x+U_i t) -\nabla_i f(x)\|_{(i)}^* \leq L_i \|t\|_{(i)}.\]

It is easy to see that if the stronger condition is satisfied, then the weaker one is also satisfied with $L_i$ no worse than
$L_i \leq \sum_{j=1}^m L_{ji}$.

\subsection{New ESO}
\label{sec:neweso}

We now derive an ESO inequality for functions satisfying Assumption~\ref{ASS:4} and $\tau$-nice sampling $\hat{S}$. That is,  $\hat{S}$ is a random subset of $[n]$ of cardinality $\tau$, chosen uniformly at random. One can derive similar bounds for all uniform samplings considered in \cite{RT:PCDM} using the same approach. 

\begin{thm} \label{thm:newESO} Let $f$ satisfy Assumption~\ref{ASS:4}. 

\begin{itemize}
\item[(i)] If $\hat{S}$ is a $\tau$-nice sampling, then for all $x,h\in \R^N$,
\begin{equation}\label{eq:PS2}\Exp \left[f(x+ h_{[\hat{S}]})\right] \leq f(x) + \frac{\tau}{n}\left(\ve{\nabla f(x)}{h} + \frac{1}{2}\|h\|_{\vv}^2\right),\end{equation}
where 
\begin{equation}\label{eq:sjhs6453}\vv_{i} \eqdef \sum_{j=1}^m \beta_j L_{ji} = \sum_{j: i \in C_j} \beta_j L_{ji}, \qquad i=1,2,\dots,n,\end{equation}
\[\beta_j \eqdef 1+ \frac{(\omega_j-1)(\tau-1)}{\max\{1,n-1\}}, \qquad j=1,2,\dots,m.\]
That is, $(f,\hat{S}) \sim ESO(\vv)$.

\item[(ii)] Moreover, for all $x,h\in \R^N$ we have
\begin{equation}\label{eq:DSO000}
f(x+h) \leq f(x) + \ve{\nabla f(x)}{h} + \frac{\bar{\omega}\bar{L}}{2}\|h\|_{w}^2,
\end{equation}
where 
\begin{equation}\label{eq:shsths} \bar{\omega} \eqdef \sum_j \omega_j \frac{\sum_i L_{ji}}{\sum_{k,i} L_{ki}}, \qquad \bar{L} \eqdef \frac{\sum_{ji} L_{ji}}{n}, \qquad w_i \eqdef \frac{n}{\sum_{j,i}\omega_j L_{ji}}\sum_j \omega_j L_{ji}.\end{equation}

Note that $\bar{\omega}$ is a data-weighted \emph{average} of the values $\{\omega_j\}$ and that $\sum w_i = n$. 
\end{itemize}

\end{thm}

\begin{proof}
Statement (ii) is a special case of (i) for $\tau=n$ (notice that $\bar{\omega}\bar{L}w = \vv$). We hence only need to prove (i). A well known consequence of \eqref{eq:newass} is 
\begin{equation}\label{eq:shs6shs}f_j(x + U_i t) \leq f_j(x) + \ve{\nabla_i f_j (x)}{t} + \frac{L_{ji}}{2}\|t\|_{(i)}^2, \qquad x \in \R^N,\; t \in \R^{N_i}.\end{equation}
We first claim that for all $i$ and $j$,
\begin{equation}\label{eq:PS1}\Exp \left[ f_j(x+ h_{[\hat{S}]})  \right] \leq f_j(x) + \frac{\tau}{n}\left(\ve{\nabla f_j(x)}{h} + \frac{\beta_j}{2}\|h\|_{L_{j:}}^2\right), \end{equation}
where $L_{j:} = (L_{j1},\dots,L_{jn}) \in \R^n$. That is, $(f_j,\hat{S}) \sim ESO(\beta_j L_{j:})$.   Equation~\eqref{eq:PS2} then follows  by adding up\footnote{At this step we could have also simply applied Theorem~10 from \cite{RT:PCDM}, which give the formula for an ESO for a conic combination of functions given ESOs for the individual functions. The proof, however, also amounts to simply adding up the inequalities.}  the  inequalities \eqref{eq:PS1} for all $j$. Let us now prove the claim.\footnote{This claim is a special case of Theorem 14 in \cite{RT:PCDM} which gives an ESO bound for a \emph{sum} of functions $f_j$ ( here we only have a single function). We include the proof as in this special case it more straightforward.} We fix $x$ and define \begin{equation}\label{eq:js0s6sh}\hat{f}_j(h) \eqdef f_j(x+h)-f_j(x)-\ve{\nabla f_j(x)}{h}.\end{equation} Since
\begin{eqnarray*}
 \Exp \left[\hat{f}_j(h_{[\Srv]})\right]
&\overset{\eqref{eq:js0s6sh}}{=}&
 \Exp \left[f_j(x+h_{[\Srv]})-f_j(x)- \ve{\nabla f_j(x)}{h_{[\Srv]}} \right]\\
 & \stackrel{\eqref{eq:0978098}}{=}&
 \Exp\left[f_j(x+h_{[\Srv]})\right]
       -f_j(x)-\tfrac{\tau}{n} \ve{\nabla f_j(x)}{h},
\end{eqnarray*}
it now only remains to show that

\begin{equation}\label{eq:aksoio4209809878}
 \Exp \left[\hat{f}_j(h_{[\Srv]})\right]\leq
 \tfrac{\tau \beta_j}{2 n}
\|h\|_{L_{j:}}^2.
\end{equation}
We now adopt the convention that  expectation conditional on an event which happens with probability 0 is equal to 0. Let $\eta_j \eqdef |C_j \cap \Srv|$, and using this convention, we can write
\begin{eqnarray}
 \Exp \left[\hat{f}_j(h_{[\Srv]})\right] 
&=&
    \sum_{k=0}^n \Prob(\eta_j = k)
      \Exp \left[\hat{f}_j(h_{[\Srv]}) \;|\; \eta_j=k\right].\label{eq:8893298d9}
\end{eqnarray}

For any $k\geq 1$ for which $\Prob(\eta_j =k)>0$,  we now use use convexity of $\hat{f}_j$ to write
\begin{eqnarray}
\Exp\left[\hat{f}_j(h_{[\Srv]}) \;|\; \eta_j = k\right]
&=&
  \Exp
    \left[\left.\hat{f}_j \left(\tfrac{1}{k} \sum_{i \in C_j \cap \Srv} k U_i h^{(i)} \right) \right. \;|\; \eta_j = k\right] \notag \\
&\leq&
  \Exp
    \left[\left.
      \tfrac{1}{k}  \sum_{i \in C_j\cap \Srv}
          \hat{f}_j \left(  k U_i h^{(i)} \right) \right. \;|\; \eta_j=k\right] \notag\\
          &=& \tfrac{1}{\omega_j}  \sum_{i \in C_j}
          \hat{f}_j \left(  k U_i h^{(i)} \right) \notag \\
          &\overset{\eqref{eq:shs6shs}+\eqref{eq:js0s6sh} }{\leq}& \tfrac{1}{\omega_j} \sum_{i \in C_j}  \tfrac{L_{ji}}{2}\|kh^{(i)}\|_{(i)}^2 \;\;= \;\;  \tfrac{k^2}{2\omega_j} \|h\|_{L_{j:}}^2\label{eq:8488dd8}
\end{eqnarray}
where the second equality follows from Equation (41) in \cite{RT:PCDM}. Finally, 
\begin{equation}
 \Exp \left[\hat{f}_j(h_{[\Srv]})\right]
\overset{\eqref{eq:8488dd8} + \eqref{eq:8893298d9}}{\leq} \sum_{k} \Prob(\eta_j = k) \tfrac{k^2}{2\omega_j} \|h\|_{L_{j:}}^2\\
= \tfrac{1}{2\omega_j} \|h\|_{L_{j:}}^2 \Exp [|C_j\cap \Srv|^2]\\
= \tfrac{\tau\beta_j}{2 n}
\|h\|_{L_{j:}}^2,
\end{equation}
where the last identity is Equation (40) in \cite{RT:PCDM}, and hence \eqref{eq:aksoio4209809878} is established.
\end{proof}

\subsection{Computation of  $L_{ji}$}

We now give  a formula for the constants $L_{ji}$ in the case when $f_j$ arises as a composition of a scalar function $\phi_j$ whose derivative has a known Lipschitz constant (this is often easy to compute), and a linear functional. 
Let $A$ be an $m\times N$ real matrix and for $j\in\{1,2,\dots,m\}$ and $i \in [n]$ define
\begin{equation} A_{ji} \eqdef e_j^T A U_i \in \R^{1 \times N_i}.\label{eq:shs7hd7d}
\end{equation}
That is, $A_{ji}$ is a row vector composed of the elements of row $j$ of $A$  corresponding to block $i$.
 
\begin{thm} \label{thm:compute_Lip}Let $f_j(x) = \phi_j(e_j^T Ax)$, where $\phi_j : \R \to \R$ is a function with  $L_{\phi_j}$-Lipschitz derivative:
\begin{equation}\label{eq:jd8djd8}|\phi_j(s)-\phi_j(s')| \leq L_{\phi_j} |s-s'|, \qquad s,s' \in \R.\end{equation}
Then $f_j$  has a block Lipshitz gradient with constants 
\begin{equation}\label{eq:jssgs5}L_{ji} = L_{\phi_j} \left(\|A_{ji}^T\|_{(i)}^*\right)^2, \qquad i=1,2,\dots,n.\end{equation}
In other words,  $f_j$ satisfies \eqref{eq:newass} with constants $L_{ji}$ given above.
\end{thm}
\begin{proof}
For any $x \in \R^N$, $t \in \R^{N_i}$ and $i$ we have
\begin{eqnarray*}
\|\nabla_i f_j(x+ U_i t) - \nabla_i f_j(x)\|_{(i)}^* &\overset{\eqref{eq:syshsj9s}}{ =} & \| U_i^T (e_j^TA)^T \phi_j'(e_j^T A(x+U_i t)) - U_i^T (e_j^TA)^T \phi_j'(e_j^T Ax)\|_{(i)}^*\\
&\overset{\eqref{eq:shs7hd7d}}{=}& \|A_{ji}^T \phi_j'(e_j^T A(x+U_i t)) -  A_{ji}^T\phi_j'(e_j^T Ax)\|_{(i)}^*\\
&\leq& \|A_{ji}^T\|_{(i)}^* |\phi_j'(e_j^T A(x+U_i t)) -  \phi_j'(e_j^T Ax)|\\
&\overset{\eqref{eq:jd8djd8}+\eqref{eq:shs7hd7d}}{\leq}& \|A_{ji}^T\|_{(i)}^* L_{\phi_j} |A_{ji} t|\;\;\leq\;\; \|A_{ji}^T\|_{(i)}^* L_{\phi_j} \|A_{ji}^T\|_{(i)}^* \|t\|_{(i)},
\end{eqnarray*}
where the last step follows by applying the  Cauchy-Schwartz inequality.
\end{proof}

\begin{example}[Quadratics] Consider the quadratic function
\[f(x) = \tfrac{1}{2}\|Ax-b\|^2 = \tfrac{1}{2}\sum_{j=1}^m (e_j^T Ax - b_j)^2.\] 
Then $f_j(x)=\phi_j(e_j^T Ax)$, where $\phi_j(s) = \tfrac{1}{2}(s-b_j)^2$ and $L_{\phi_j}=1$. 

\begin{enumerate}
\item[(i)] Consider the block setup with $N_i=1$ (all blocks are of size 1) and $B_i=1$ for all $i \in [n]$.  Then $L_{ji}=A_{ji}^2$. In Table~\ref{tbl:more_stepsizes} we list stepsizes for coordinate descent methods proposed in the literature. It can be seen that our stepsizes are better than those proposed by Richt\'{a}rik \& Tak\'{a}\v{c} \cite{RT:PCDM} and those proposed by Necoara \& Clipici \cite{NC-dist}. Indeed, $\vv^{\text{rt}}_i\geq \vv^{\text{fr}}_i$ for all $i$. The difference grows as $\tau$ grows; and there is equality for $\tau=1$. We also have $\|\vv^{\text{nc}}\|_1\geq \|\vv^{\text{fr}}\|_1$, but here the difference decreases with $\tau$; and there is equality for $\tau=n$.

\item[(ii)] Choose nontrivial block sizes
 and define data-driven block norms with $B_i=A_i^T A_i$, where $A_i = A U_i$, assuming that the matrices $A_i^T A_i$ are positive definite. Then 
 \[L_{ji} = L_{\phi_j}(\|A_{ji}^T\|_{(i)}^*)^2 \overset{\eqref{eq:norm_block_conj}}{=} \ve{(A_i^T A_i)^{-1} A_{ji}^T}{A_{ji}^T} \overset{\eqref{eq:shs7hd7d}}{=} e_j^T A_i (A_{i}^T A_i)^{-1} A_i^T e_j.\]
\end{enumerate}
\end{example}

Table~\ref{tbl:lip} lists  constants $L_\phi$ for selected scalar loss functions $\phi$ popular in machine learning.

\begin{table}
\begin{center}
\begin{tabular}{|l|l|c|}
\hline
Loss & $\phi(s)$ & $L_\phi$\\
\hline
\phantom{-} &  &\\
Square Loss & $\tfrac{1}{2}s^2$ & 1\\
\phantom{-} & &\\
Logistic Loss & $\log(1+e^s)$ & 1/4\\
\phantom{-} & & \\
\hline
\end{tabular}
\end{center}

\caption{\footnotesize Lipschitz constants of the derivative of selected scalar loss  functions.}
\label{tbl:lip}
\end{table}


\begin{table}
\begin{center}
\begin{tabular}{|l|c|}
\hline
Paper & $\vv_i$\\
\hline
\phantom{-} & \\
Richt\'{a}rik \& Tak\'{a}\v{c} \cite{RT:PCDM} & $\vv^{\text{rt}}_i=\sum_{j=1}^m \left(1+\tfrac{(\omega-1)(\tau-1)}{\max\{1,n-1\}}\right) A_{ji}^2$ \\
\phantom{-} & \\
Necoara \& Clipici \cite{NC-dist}  & $\vv^{\text{nc}}_i = \sum_{j: i \in C_j} \sum_{k=1}^n A_{jk}^2$\\
\phantom{-} & \\
This paper & $\vv^{\text{fr}}_i = \sum_{j=1}^m \left(1 + \frac{(\omega_j-1)(\tau-1)}{\max\{1,n-1\}}\right) A_{ji}^2$ \\
\phantom{-} & \\
\hline
\end{tabular}
\end{center}

\caption{\footnotesize ESO stepsizes for coordinate descent methods suggested in the literature in the case of a quadratic $f(x)=\tfrac{1}{2}\|Ax-b\|^2$. We consider setup with elementary block sizes ($N_i=1$) and $B_i=1$.}
\label{tbl:more_stepsizes}
\end{table}

\section{Accelerated parallel coordinate descent} \label{SEC:ACD}

We are interested in solving the regularized optimization problem
\begin{equation}\label{eq:MAIN1}
\begin{aligned}
\text{minimize} \quad & F(x) \eqdef f(x) + \Reg(x) ,\\
\text{subject to} \quad & x=(x^{(1)},\dots,x^{(n)})\in \R^{N_1}\times \cdots \times \R^{N_n} = \R^N,
\end{aligned}
\end{equation}
where  $\Reg: \R^N \to \R\cup \{+\infty\}$ is a (possibly nonsmooth) convex regularizer that is separable in the blocks $x^{(i)}$:
\begin{equation}\label{eq:MAIN5}\Reg(x) = \sum_{i=1}^n \Reg_i(x^{(i)}).\end{equation}

\subsection{The algorithm}

We now describe our method (Algorithm~\ref{algo:initial}). It is presented here in a form that facilitates analysis and comparison with existing methods. In Section~\ref{SEC:EFFICIENT_IMPLEMENTATION} we rewrite the method into a different (equivalent) form -- one that is geared towards practical efficiency.

\begin{algorithm}
\begin{algorithmic}[1]
\STATE Choose $x_0 \in \R^N$ and set $z_0=x_0$ and $\theta_0=\frac{\tau}{n}$
\FOR{ $k \geq 0$}
\STATE $y_k= (1-\theta_k)x_k+ \theta_k z_k$
\STATE Generate a random set of blocks $S_k\sim \hat{S}$
\STATE $z_{k+1} = z_k$
\FOR{ $i \in S_k$}
\STATE $z_{k+1}^{(i)} = \arg\min_{z \in \R^{N_i}} \left\{ \langle \nabla_i f (y_k), z - y_k^{(i)} \rangle + \frac{n\theta_k  \vv_i}{2\tau} \|z-z_k^{(i)}\|_{(i)}^2 + \Reg_i(z)\right\}$
\ENDFOR 
\STATE $x_{k+1} =y_k +\frac{n}{\tau}\theta_k(z_{k+1}-z_k)$
\STATE $\theta_{k+1}=\frac{\sqrt{\theta_k^4+4 \theta_k^2} - \theta_k^2}{2}$
\ENDFOR 
\end{algorithmic}
\caption{APPROX: {\bf A}ccelerated {\bf P}arallel {\bf Prox}imal Coordinate Descent Method}
\label{algo:initial}
\end{algorithm}

\bigskip

The method starts from $x_0\in \R^N$ and generates three vector sequences, $\{x_k,y_k,z_k\}_{k\geq 0}$. 
In Step~3, $y_k$ is defined as a convex combination of $x_k$ and $z_k$, which may in general be full dimensional vectors. This is not efficient; but we will ignore this issue for now. In Section~\ref{SEC:EFFICIENT_IMPLEMENTATION} we show that it is possible to implement the method in such a way that it not necessary to ever form $y_k$. In Step 4 we generate a random block sampling $S_k$ and then perform steps 5--9 in parallel. The assignment $z_{k+1}\leftarrow z_k$ is not necessary in practice; the vector $z_k$ should be overwritten in place. Instead, Steps 5--8 should be seen as saying that we update blocks $i \in S_k$ of $z_k$, by solving $|S_k|$ proximal problems in parallel,  and call the resulting vector $z_{k+1}$.
Note in Step 9,  $x_{k+1}$ should also be computed in parallel. Indeed,  $x_{k+1}$ is obtained from $y_k$ by changing the blocks of $y_k$ that belong to $S_k$ - this is because $z_{k+1}$ and $z_k$ differ in those blocks only. Note that gradients are evaluated only at $y_k$. We  show in Section~\ref{SEC:EFFICIENT_IMPLEMENTATION} how this can be done efficiently, for some problems, without the need to form $y_k$.

We now formulate the main  result of this paper.

\begin{thm}\label{thm:main}
Let Assumption~\ref{ass:ESO} hold, with $(f, \hat{S}) \sim \mathrm{ESO}(\vv)$, where $\tau=\Exp[|\hat{S}|]>0$. Let $x_0 \in \dom \Reg$, and assume that the random sets $S_k$ in Algorithm~\ref{algo:initial} are chosen independently, following the distribution of $\hat{S}$. Then for any optimal point $x_*$ of problem \eqref{eq:MAIN1}, the iterates $\{x_k\}_{k\geq 1}$ of Algorithm~\ref{algo:initial} satisfy
\begin{eqnarray}\label{eq:main_thm077}
\Exp[F(x_k)-F(x_*)] &\leq& \frac{4 n^2}{((k-1) \tau+2n)^2} C,
\end{eqnarray}
where
\begin{eqnarray}\label{eq:sjs8hd876d}
C &\eqdef& \left(1-\frac{\tau}{n}\right)(F(x_0)-F(x_*))+ \frac{ 1}{2}  \norm{x_0-x_*}_{\vv}^2. 
\end{eqnarray}
In other words, for any $0<\epsilon\leq C$, the number of iterations for obtaining an $\epsilon$-solution in expectation does not exceed
\begin{eqnarray}\label{eq:k_0000}
k &=& \left\lceil \frac{2 n }{\tau } \left(\sqrt{\frac{C}{\epsilon}} - 1\right)+1\right\rceil.
\end{eqnarray}
\end{thm}

The proof of Theorem~\ref{thm:main} can be found in Section~\ref{SEC:COMPLEXITY}. We now comment on the result:

\begin{enumerate}
\item Note that we do not assume that $f$ be of the form \eqref{eq:main-intro}; all that is needed is Assumption~\ref{ass:ESO}.
\item If $n=1$, we recover Tseng's proximal gradient descent algorithm~\cite{tseng2008accelerated}. If $n>1$, $\tau=1$ and $\Reg\equiv 0$, we obtain a new version of (serial) accelerated coordinate descent~\cite{Nesterov:2010RCDM, lee2013efficient} for minimizing smooth functions. Note that no existing accelerated coordinate descent methods are either proximal, \emph{or} parallel. Our method is both proximal and parallel.


\item In the case when we update all blocks in one iteration ($\tau=n$), the bound \eqref{eq:main_thm077} simplifies to
\begin{equation}\label{eq:sks9s6s5}F(x_k)-F(x_*) \leq \frac{2  \frac{\|\vv\|_1}{n}}{(k+1)^2}\|x_0-x_*\|_{\tilde{\vv}}^2,\end{equation}
where as before, $\tilde{\vv}=n \vv/\|\vv\|_1$.
There is no expectation here as the method is deterministic in this case.

If we use stepsize $\vv$ proposed in Theorem~\ref{thm:newESO}, then in view of part (ii) of that theorem, bound \eqref{eq:sks9s6s5} takes the form
\begin{equation}\label{eq:sks9s6s50909}F(x_k)-F(x_*) \leq \frac{2 \bar{\omega}\bar{L}}{(k+1)^2}\|x_0-x_*\|_{w}^2,\end{equation}
as advertised in the abstract. Recall that $\bar{\omega}$ is a data-weighted \emph{average} of the values $\{\omega_j\}$.

In contrast, using the stepsizes proposed by Richt\'{a}rik \& Tak\'{a}\v{c} \cite{RT:PCDM} (see Table~\ref{tbl:more_stepsizes}), we get

\begin{equation}\label{eq:sks9s6s59985}F(x_k)-F(x_*) \leq \frac{2 \omega \frac{\sum_i L_{i}}{n}}{(k+1)^2}\|x_0-x_*\|_{\tilde{\vv}}^2.\end{equation}


Note that in the case when the functions $f_j$ are convex quadratics ($f_j(x)=\tfrac{1}{2}(a_j^T x-b_j)^2$), for instance, we have $L_i = \sum_{j} L_{ji}$, and hence the new ESO leads to a vast improvement in the complexity in cases when $\bar{\omega}\ll \omega$. On he other hand, in cases where $L_i \ll \sum_{j} L_{ji}$ (which can happen with logistic regression, for instance), the result based on the Richt\'{a}rik-Tak\'{a}\v{c} stepsizes \cite{RT:PCDM} may be better.

\item  Consider the smooth case ($\Reg\equiv 0$): $F=f$ and $f'(x_*)=0$.  By part (ii) of Theorem~\ref{thm:newESO}, $\nabla f$ is Lipschitz with  constant $1$ wrt $\|\cdot\|_{w}$.  Choosing $x=x_*$ and $h = x_0-x_*$, we get
\begin{equation}\label{eq:sjsushd0000}f(x_0)-f(x_*)\leq \frac{1}{2}\|x_0-x_*\|_{w}^2.\end{equation}

Now, consider running Algorithm~\ref{algo:initial} with a $\tau$-nice sampling and stepsize parameter $\vv$ as in Theorem~\ref{thm:newESO}. Letting $d=(d_1,\dots,d_n)$, where $d_i$ is defined by
\begin{equation}\label{eq:shsrps0} \left(1-\frac{\tau}{n}\right)w_i + \vv_i = \left(1-\frac{\tau}{n}\right)\sum_j \omega_j L_{ji} + \sum_{j} \beta_j L_{ji} \leq \sum_{j} (\omega_j+1)L_{ji} \eqdef d_i,\end{equation}
we get
\begin{eqnarray*}
\Exp[f(x_k)-f(x_*)] &\overset{\eqref{eq:main_thm077}+\eqref{eq:sjsushd0000}}{ \leq} & \frac{2n^2}{((k-1)\tau +2n)^2}\|x_0-x_*\|_{\left(1-\frac{\tau}{n}\right)w + \vv}^2 \\
&\overset{\eqref{eq:shsrps0}}{\leq}& \frac{2n^2}{((k-1)\tau +2n)^2}\|x_0-x_*\|^2_d \\ &
\overset{\eqref{eq:sjhs6453}+\eqref{eq:shsths}}{\leq}& \frac{2n^2 (\bar{\omega}+1)\bar{L}}{((k-1)\tau +2n)^2} \|x_0-x_*\|_{\tilde{d}}^2,\end{eqnarray*}
where   in the last step we have used the estimate $\omega_j +\beta_j - \tfrac{\tau \omega_j}{n} \in [\omega_j,\omega_j+1]$, and $\tilde{d}$ is a scalar multiple of $d$ for which $\|\tilde{d}\|_1=1$. Similarly as in \eqref{eq:k_0000}, this means that
\[k  \geq k(\tau) \eqdef 1+\frac{n}{\tau}\sqrt{\frac{2(\bar{\omega}+1)\bar{L}}{ \epsilon }} \|x_0-x_*\|_{\tilde{d}}\]
iterations suffice to produce an $\epsilon$-solution in expectation. Hence, we get \emph{linear speedup} in the number of parallel updates / processors. 
This is \emph{different} from the situation in simple (non-accelerated) parallel coordinate descent methods where parallelization speedup depends on the degree of separability (speedup is better if $\omega$ is small). In APPROX, the average degree of separability $\bar{\omega}$ is decoupled from $\tau$, and hence one benefits from separability even for large $\tau$. This means that \emph{accelerated methods are  more suitable for parallelization}.

\item
We focused on the case of \emph{uniform} samplings, but with a proper change in the definition of  ESO, one can also handle 
\emph{non-uniform} samplings~\cite{RT:2013optimal}.
\end{enumerate}


\section{Complexity analysis}\label{SEC:COMPLEXITY}

We first establish four lemmas and then prove Theorem~\ref{thm:main}.

\subsection{Lemmas}

In the first lemma we summarize well-known properties of the sequence $\theta_k$ used in Algorithm~\ref{algo:initial}.

\begin{lem}[Tseng \cite{tseng2008accelerated}]
\label{lem:theta}
The sequence $\{\theta_k\}_{k\geq 0}$ defined in Algorithm~\ref{algo:initial} is decreasing and satisfies $0 < \theta_k \leq \frac{2}{k+2n / \tau} \leq \tfrac{\tau}{n}\leq 1$ and \begin{equation}\label{eq:theta_id}\frac{1-\theta_{k+1}}{\theta_{k+1}^2} = \frac{1}{\theta_k^2}.
\end{equation} 
\end{lem}

We now give an explicit characterization of $x_k$ as a convex combination of the vectors $z_0,\dots,z_k$. 

\begin{lem}
\label{lem:xconvcombz} Let $\{x_k,z_k\}_{k\geq 0}$ be the iterates of Algorithm~\ref{algo:initial}. Then for all $k \geq 0$ we have 
\begin{equation}\label{eq:x_k-convex}
x_{k} = \sum_{l=0}^{k} \gamma_k^l z_l,
\end{equation}
 where the coefficients $\gamma_{k}^0,\gamma_{k}^1,\dots,\gamma_{k}^{k}$ are non-negative and sum to 1. That is,  $x_k$ is a convex combination of the vectors $z_0,z_1,\dots,z_{k}$. In particular,
 the constants  are defined recursively in $k$ by setting $\gamma_{0}^0=1$, $\gamma_{1}^0 = 0$, $\gamma_{1}^1 = 1$ and for $k\geq 1$,
\begin{equation}\label{eq:gammas}\gamma_{k+1}^l = \begin{cases}
(1-\theta_{k})\gamma_{k}^l, & l = 0,\dots,k-1,\\
\theta_{k} (1-\frac{n}{\tau}\theta_{k-1})+\frac{n}{\tau}(\theta_{k-1}-\theta_{k}),  & l=k,\\
\tfrac{n}{\tau}\theta_{k}, & l=k+1. 
\end{cases}\end{equation}
Moreover, for all $k\geq 0$, the following identity holds
\begin{equation}\label{eq:s9djd7}\gamma_{k+1}^k + \frac{n-\tau}{\tau}\theta_k  = (1-\theta_k)\gamma_{k}^k.\end{equation}

\end{lem}

\begin{proof}
We proceed by induction. First, notice that $x_0=z_0 = \gamma_0^0 z_0$.  This  implies that $y_0=z_0$,  which in turn together with $\theta_0=\tfrac{\tau}{n}$ gives
$x_1=y_0 + \tfrac{n}{\tau}\theta_0(z_1-z_0) = z_1 = \gamma_{1}^0 z_0 + \gamma_{1}^1 z_1$. Assuming now that \eqref{eq:x_k-convex} holds for some $k\geq 1$, we obtain
\begin{eqnarray}
x_{k+1}&\overset{\text{(Alg~\ref{algo:initial}, step 9)}}{=}& y_k+\frac{n}{\tau}\theta_k(z_{k+1}-z_k) \notag\\
&\overset{\text{(Alg~\ref{algo:initial}, step 3)}}{=}&(1-\theta_k) x_k + \theta_k z_k -\frac{n}{\tau}\theta_k z_k + \frac{n}{\tau} \theta_k z_{k+1}\label{eq:affine} \\
&=&   \sum_{l=0}^{k-1} \underbrace{(1-\theta_k) \gamma_{k}^l}_{\gamma_{k+1}^l} z_l  +  \underbrace{\left((1-\theta_k)\gamma_{k}^k  + \theta_k  - \frac{n}{\tau}\theta_k \right)}_{\gamma_{k+1}^k} z_k 
+  \underbrace{\left( \frac{n}{\tau} \theta_k \right)}_{\gamma_{k+1}^{k+1}} z_{k+1}.\notag
\end{eqnarray}
By applying Lemma~\ref{lem:theta}, together with the inductive assumption that $\gamma_{k}^l\geq 0$ for all $l$,  we observe that $\gamma_{k+1}^l\geq 0$ for all $l$. It remains to show that the constants sum to 1. This is true since $x_{k}$ is a convex combination of $z_1,\dots,z_k$, and by \eqref{eq:affine}, $x_{k+1}$ is an affine combination of $x_k$, $z_k$ and $z_{k+1}$. 
\end{proof}

Define
\begin{eqnarray*}
\tilde{z}_{k+1} &\eqdef& \arg \min_{z \in \R^N} \left\{\Reg(z) + \ve{\nabla f(y_k)}{z-y_k} + \frac{n\theta_k}{2\tau}\|z-z_k\|_{\vv}^2\right\}\\
&\overset{\eqref{eq:norm_block}+\eqref{eq:MAIN5}}{=}& \arg \min_{z = (z^{(1)},\dots,z^{(n)}) \in \R^N} \sum_{i=1}^n\left\{ \Reg_i(z^{(i)}) + \ve{\nabla_i f(y_k)}{z^{(i)}-y_k^{(i)}} + \frac{n\theta_k \vv_i }{2\tau}\|z^{(i)}-z_k^{(i)}\|_{(i)}^2\right\}.
\end{eqnarray*}
From this and the definition of $z_{k+1}$ we see that  \begin{equation}\label{eq:z_{k+1}}
z_{k+1}^{(i)}= \begin{cases}
\tilde{z}_{k+1}^{(i)}, & i \in S_k\\
z_{k}^{(i)}, & i \not \in S_k.\\
\end{cases}
\end{equation}

The next lemma is an application to a specific function of a well-known result that can be found, for instance, in \cite{tseng2008accelerated}. The result was used  by Tseng  to construct a simplified complexity proof for a proximal gradient descent method. This lemma requires the norms $\|\cdot\|_{(i)}$ to be Euclidean -- and this is the only place in our analysis where this is required.

\begin{lem}[Property 1 in \cite{tseng2008accelerated}]
 Let $\xi(u) \eqdef f(y_k)+ \ve{\nabla f(y_k)}{u - y_k} + \frac{n\theta_k }{2\tau}  \|u - z_k\|_{\vv}^2$. Then
\begin{equation}\label{eq:9sjs8s}
\Reg(\tilde{z}_{k+1})+ \xi(\tilde{z}_{k+1}) \leq \Reg(x_*)+ \xi(x_*) - \frac{n\theta_k }{2\tau} \norm{x_*-\tilde{z}_{k+1}}_{\vv}^2.
\end{equation}
\end{lem}

Our next lemma is a technical result connecting the gradient mapping (producing $\tilde{z}_{k+1}$) and the stochastic block gradient mapping (producing the random vector $z_{k+1}$). The lemma reduces to a trivial identity in the case when of a single block ($n=1$). From now on, by $\Exp_k$ we denote the expectation with respect to $S_k$, keeping everything else fixed.

\begin{lem} For any  $x \in \R^N$ and $k\geq 0$,
\begin{equation}\label{eq:99sjs7ss}
\Exp_k \left[\|z_{k+1}-x\|_{\vv}^2 - \|z_k-x\|_{\vv}^2\right] = \frac{\tau}{n} \left(\|\tilde{z}_{k+1}-x\|_{\vv}^2 - \|z_k-x\|_{\vv}^2\right).
\end{equation}
Moreover,
\begin{equation}\label{eq:separable_exp}\Exp_k\left[\Reg(z_{k+1})\right] = \left(1-\frac{\tau}{n}\right)\Reg(z_k) + \frac{\tau}{n}\Reg(\tilde{z}_{k+1}).\end{equation}
\end{lem}

\begin{proof} Let $\hat{S}$ be any uniform sampling and $a,h \in \R^N$. 
 Theorem~4 in \cite{RT:PCDM} implies that
\begin{equation}\label{eq:0978098}\Exp[\|h_{[\hat{S}]}\|_{\vv}^2] = \tfrac{\tau}{n}\|h\|_{\vv}^2, \quad \Exp[\ve{a}{h_{[\hat{S}]}}_{\vv}] = \tfrac{\tau}{n} \ve{a}{h}_{\vv}, \quad \Exp[\Reg(a+h_{[\hat{S}]})] = \left(1-\tfrac{\tau}{n}\right)\Reg(a) + \tfrac{\tau}{n}\Reg(a+h),\end{equation}
where  $\ve{a}{h}_{\vv}\eqdef \sum_{i=1}^n {\vv}_i\ve{a^{(i)}}{h^{(i)}}$. Let $h = \tilde{z}_{k+1} - z_k$. In view of \eqref{eq:h_S} and \eqref{eq:z_{k+1}}, we can write $z_{k+1}-z_k  = h_{[S_k]}$. Applying the first two identities in \eqref{eq:0978098} with $a=z_k-x$ and $\hat{S}=S_k$, we get
\begin{eqnarray*}
\Exp_k \left[\|z_{k+1}-x\|_{\vv}^2 - \|z_k-x\|_{\vv}^2\right] &=& \Exp_k \left[\|h_{[S_k]}\|_{\vv}^2 + 2 \ve{z_k-x}{h_{[S_k]}}_{\vv}\right]\\
&\overset{\eqref{eq:0978098}}{=} & \frac{\tau}{n}\left(\|h\|_{\vv}^2 + 2 \ve{z_k-x}{h}_{\vv}\right)\;\;=\;\;\frac{\tau}{n} \left(\|\tilde{z}_{k+1}-x\|_{\vv}^2 - \|z_k-x\|_{\vv}^2\right).
\end{eqnarray*}
The remaining statement follows from the last identity in \eqref{eq:0978098} used with $a = z_k$.
\end{proof}

\subsection{Proof of Theorem~\ref{thm:main}}

Using Lemma~\ref{lem:xconvcombz} and convexity of $\Reg$, for all $k\geq 0$ we have
\begin{equation}
\label{eq:js8snss0}\Reg(x_{k}) \;\;\overset{\eqref{eq:x_k-convex}}{=}\;\;\Reg\left(\sum_{l=0}^{k} \gamma_{k}^l z_l\right)\;\; \overset{(\text{convexity})}{\leq}\;\; \sum_{l=0}^{k}\gamma_{k}^l \Reg(z_l)\;\; \eqdef \;\; \hat{\Reg}_{k}. \end{equation} 
From this we get
\begin{eqnarray} 
\Exp_k[\hat{\Reg}_{k+1}]
&\overset{\eqref{eq:js8snss0}+\eqref{eq:gammas}}{=}& \sum_{l=0}^{k} \gamma_{k+1}^l \Reg(z_l) +  \frac{n}{\tau} \theta_k \Exp_k \left[\Reg(z_{k+1})\right]\notag\\
&\overset{\eqref{eq:separable_exp}}{=} &
\sum_{l=0}^{k} \gamma_{k+1}^l \Reg(z_l) +  \frac{n}{\tau} \theta_k \left(\left(1-\frac{\tau}{n}\right)\Reg(z_k) + \frac{\tau}{n}\Reg(\tilde{z}_{k+1})\right)\notag\\
&=&
\sum_{l=0}^{k} \gamma_{k+1}^l \Reg(z_l) +   \left(\frac{n}{\tau}-1\right)\theta_k\Reg(z_k) + \theta_k\Reg(\tilde{z}_{k+1}).\label{eq:js8s8ss}
\end{eqnarray}

Since  $x_{k+1}=y_k+h_{[S_k]}$ with $h=\frac{n}{\tau}\theta_k (\tilde{z}_{k+1} - z_k)$, we can use  ESO to bound
\begin{eqnarray}
\Exp_k[f(x_{k+1})]& \overset{\eqref{eq:ESO}}{\leq}& f(y_k) + \theta_k \langle \nabla f(y_k) , \tilde{z}_{k+1} - z_k \rangle + \frac{n\theta_k^2}{2\tau}\norm{\tilde{z}_{k+1} - z_k}_{\vv}^2 \notag\\
& = &(1-\theta_k) f(y_k)  -  \theta_k \langle \nabla f(y_k) , z_k - y_k \rangle \notag\\
&& \;\; + \theta_k \Big( f(y_k) + \langle \nabla f(y_k) , \tilde{z}_{k+1} - y_k \rangle + \frac{n\theta_k}{2\tau}\norm{\tilde{z}_{k+1} - z_k}_{\vv}^2  \Big). \label{eq:sjsuhs6s}
\end{eqnarray}

Note that from the definition of $y_k$ in the algorithm, we have

\begin{equation}\label{eq:shbd4d7d}
\theta_k (y_k - z_k) = ((1-\theta_k)x_k - y_k) + \theta_k y_k = (1-\theta_k)(x_k-y_k).
\end{equation}

For all $k\geq 0$ we define  an upper bound on $F(x_k)$,
 \begin{equation}\label{eq:shd7dhd}\hat{F}_{k} \;\;\eqdef\;\; \hat{\Reg}_{k} + f(x_{k}) \;\; \overset{\eqref{eq:js8snss0}}{\geq} \;\; F(x_{k}),\end{equation} 
and bound the  expectation of $\hat{F}_{k+1}$ in $S_k$ as follows:

\begin{eqnarray}
\Exp_k[\hat{F}_{k+1}]
&=&\Exp_k[\hat{\Reg}_{k+1}] + \Exp_k[f(x_{k+1})]\notag\\
&\overset{\eqref{eq:js8s8ss}+\eqref{eq:sjsuhs6s}}{\leq} &\sum_{l=0}^{k} \gamma_{k+1}^l \Reg(z_l) +  \frac{n-\tau}{\tau} \theta_k \Reg(z_{k})  +(1-\theta_k) f(y_k)  -  \theta_k \langle \nabla f(y_k) , z_k - y_k \rangle \notag\\
& \quad& + \theta_k \Big(\Reg(\tilde{z}_{k+1}) + f(y_k) + \langle \nabla f(y_k) , \tilde{z}_{k+1} - y_k \rangle + \frac{n \theta_k }{2\tau}\norm{\tilde{z}_{k+1} - z_k}_{\vv}^2  \Big) \notag\\
&\overset{\eqref{eq:9sjs8s}}{\leq} &  \sum_{l=0}^{k} \gamma_{k+1}^l \Reg(z_l) +  \frac{n-\tau}{\tau} \theta_k \Reg(z_{k})+(1-\theta_k) f(y_k)  -  \theta_k \langle \nabla f(y_k) , z_k - y_k \rangle \notag\\
& \quad & + \theta_k \Big( \Reg(x_*)+f(y_k)+\langle \nabla f(y_k) , x_* - y_k \rangle + \frac{n\theta_k }{2\tau}  \norm{x_*-z_k}_{\vv}^2 - \frac{n\theta_k }{2\tau} \norm{x_*-\tilde{z}_{k+1}}_{\vv}^2 \Big) \notag\\
&\overset{\eqref{eq:shbd4d7d}}{=}&   \sum_{l=0}^{k-1} \underbrace{\gamma_{k+1}^l}_{\overset{\eqref{eq:gammas}}{=}(1-\theta_k)\gamma_{k}^l} \Reg(z_l) +  \underbrace{\left(\gamma_{k+1}^k + \frac{n-\tau}{\tau}\theta_k\right) }_{ \overset{\eqref{eq:s9djd7}}{=} (1-\theta_k)\gamma_{k}^k}\Reg(z_k) \notag\\
&\quad& + \underbrace{(1-\theta_k) f(y_k)  +(1-\theta_k) \langle \nabla f(y_k) , x_k - y_k \rangle}_{\leq (1-\theta_k)f(x_k)} \notag \\
& \quad& + \theta_k \Big( \underbrace{\Reg(x_*)+f(y_k)+\langle \nabla f(y_k) , x_* - y_k \rangle}_{\leq F(x_*)} + \frac{n\theta_k }{2\tau}  \norm{x_*-z_k}_{\vv}^2
- \frac{n\theta_k }{2\tau} \norm{x_*-\tilde{z}_{k+1}}_{\vv}^2 \Big)\notag \\
&\overset{\eqref{eq:js8snss0}+\eqref{eq:shd7dhd}}{\leq} & (1-\theta_k) \hat{F}_k + \theta_k F(x_*) + \frac{n\theta_k^2 }{2\tau} \left( \norm{x_*-z_k}_{\vv}^2
-\norm{x_*-\tilde{z}_{k+1}}_{\vv}^2 \right) \notag\\
& \overset{\eqref{eq:99sjs7ss}}{=}&  (1-\theta_k) \hat{F}_k + \theta_k F(x_*) + \frac{n^2\theta_k^2 }{2\tau^2} \left( \norm{x_*-z_k}_{\vv}^2
-\Exp_k \left[\norm{x_*-z_{k+1}}_{\vv}^2\right]\right). \label{eq:msg876}
\end{eqnarray}


After dividing both sides of \eqref{eq:msg876} by $\theta_k^2$, using \eqref{eq:theta_id}, and rearranging the terms, we obtain 
\[
\frac{1 -\theta_{k+1}}{\theta_{k+1}^2}\Exp_k[\hat{F}_{k+1}-F(x_*)] + \frac{n^2}{2\tau^2} \Exp_k [\norm{x_*-z_{k+1}}_{\vv}^2] \;\;\leq \;\; \frac{1-\theta_k}{\theta_k^2} (\hat{F}_k-F(x_*)) + \frac{n^2}{2\tau^2} \norm{x_*-z_k}_{\vv}^2 .
\]

We now apply total expectation to the above inequality and unroll the recurrence for $l$ between 0 and $k$, obtaining
\begin{eqnarray}
\frac{1 -\theta_{k}}{\theta_{k}^2}\Exp[\hat{F}_{k}-F(x_*)] + \frac{n^2}{2\tau^2}\Exp[\|x_*-z_{k+1}\|_{\vv}^2] &\leq& \frac{1-\theta_0}{\theta_0^2} (\hat{F}_0-F(x_*))+ \frac{n^2}{2\tau^2} \norm{x_*-z_0}_{\vv}^2, \label{eq:mnb8733}
\end{eqnarray}
from which we finally get for $k \geq 1$,
\begin{eqnarray*}
\Exp[F(x_k)-F(x_*)]&\overset{\eqref{eq:shd7dhd}}{\leq}& \Exp[\hat{F}_{k}-F(x_*)]\\
 &\overset{\eqref{eq:mnb8733}}{\leq} & \frac{\theta_{k-1}^2}{\theta_0^2} (1-\theta_0)(\hat{F}_0-F(x_*))+ \frac{n^2 \theta_{k-1}^2}{2\tau^2}  \norm{x_*-z_0}_{\vv}^2  \\
& \leq& \frac{4 n^2}{((k-1) \tau+2n)^2} \Big( \left(1-\frac{\tau}{n}\right)(F(x_0)-F(x_*))+ \frac{1}{2}  \norm{x_0-x_*}_{\vv}^2 \Big), 
\end{eqnarray*}
where in the last step we have used the facts that $\hat{F}_0= F(x_0)$, $x_0=z_0$, $\theta_0=\frac{\tau}{n}$ and the estimate $\theta_{k-1} \leq \tfrac{2}{k-1+2n/\tau}$ from Lemma~\ref{lem:theta}.

\section{Implementation without full-dimensional vector operations}\label{SEC:EFFICIENT_IMPLEMENTATION}

Algorithm~\ref{algo:initial}, as presented, performs full-dimensional vector operations. Indeed, $y_k$ is defined as a convex combination of $x_k$ and $z_k$. Also, $x_{k+1}$ is obtained from $y_k$ by changing $|S_k|$ coordinates; however, if $|S_k|$ is small, the latter operation is not costly. In any case, vectors $x_k$ and $z_k$ will in general be dense, and  hence computation of $y_k$ may cost $O(N)$ arithmetic operations. However, simple (i.e., non-accelerated) coordinate descent methods are successful and popular precisely because they can avoid such operations.

Borrowing ideas from Lee \& Sidford \cite{lee2013efficient}, we rewrite\footnote{Note that we override the notation $\tilde{z}_k$ here -- it now has a different meaning from that in Section~\ref{SEC:COMPLEXITY}.
}  Algorithm~\ref{algo:initial} into a new form,  incarnated as Algorithm~\ref{algo:w-simple}. Note that the notation $\tilde{z}_k$ used here has a different meaning than in the previous section.

\begin{algorithm}
\begin{algorithmic}[1]
\STATE Pick $\tilde{z}_0\in \R^N$ and set $\theta_0=\frac{\tau}{n}$, $\pp_0=0$
\FOR{ $k \geq 0$}
\STATE Generate a random set of blocks $S_k \sim \hat{S}$
\STATE $\pp_{k+1} \leftarrow \pp_k$, $\tilde{z}_{k+1} \leftarrow \tilde{z}_k$ 
\FOR{ $i \in S_k$}
\STATE $t_{k}^{(i)} = \arg\min_{t \in \R^{N_i}} \left\{ \langle \nabla_i f (\theta_k^2 \pp_k + \tilde{z}_k), t \rangle + \frac{n\theta_k  \vv_i}{2\tau}  \|t\|_{(i)}^2 + \Reg_i(\tilde{z}_{k}^{(i)}+t)\right\}$ 
\STATE $\tilde{z}_{k+1}^{(i)} \leftarrow \tilde{z}_{k}^{(i)} + t_{k}^{(i)}$
\STATE $\pp_{k+1}^{(i)} \leftarrow \pp_k^{(i)} - \frac{1-\frac{n}{\tau} \theta_k}{\theta_k^2}t_{k}^{(i)}$
\ENDFOR 
\STATE $\theta_{k+1}=\frac{\sqrt{\theta_k^4+4 \theta_k^2} - \theta_k^2}{2}$
\ENDFOR 
\STATE OUTPUT: $\theta_{k}^2 \pp_{k+1} + \tilde{z}_{k+1}$
\end{algorithmic}
\caption{APPROX (written in a form  facilitating efficient implementation) }
\label{algo:w-simple}
\end{algorithm}

Note that if instead of updating the constants $\theta_k$ as in line 10 we keep them constant throughout, $\theta_k=\tfrac{\tau}{n}$, then $u_k=0$ for all $k$. The resulting method is precisely the PCDM algorithm (\emph{non-accelerated} parallel block-coordinate descent method) proposed and analyzed in \cite{RT:PCDM}. 

As it is not immediately obvious that the two methods(Algorithms 1 and 2) are equivalent, we include the following result. Its proof can be found in the appendix.

\begin{prop}[Equivalence] \label{prop:equivalence2} 
Run Algorithm~\ref{algo:w-simple} with $\tilde{z}_0=x_0$, where $x_0\in \dom \Reg$ is the starting point of Algorithm~\ref{algo:initial}. If we define
\begin{equation}\label{eq:xxx}
\tilde{x}_{k} = \begin{cases} 
\tilde{z}_0, & \quad k=0,\\
\theta_{k-1}^2 \pp_{k} + \tilde{z}_k, & \quad k\geq 1,
\end{cases}
\end{equation} 
and
\begin{equation}\label{eq:yyyy} \tilde{y}_k = \theta_k^2 \pp_k + \tilde{z}_k, \quad k\geq 0, \end{equation}
then $x_k=\tilde{x}_k$, $y_k=\tilde{y}_k$ and $z_k = \tilde{z}_k$ for all $k\geq 0$. That is, Algorithms~\ref{algo:initial} and \ref{algo:w-simple} are equivalent.
\end{prop}

Note that in Algorithm~\ref{algo:w-simple} we never need to form $x_k$ throughout the iterations. The only time this is needed is when producing the output: $x_{k+1} = \theta_k^2 \pp_{k+1} + z_{k+1}$.    More importantly, note that the method
does need to  explicitly compute $y_k$. Instead, we introduce a new vector, $\pp_k$, and express $y_k$ as $y_k = \theta_k^2 \pp_k + \tilde{z}_k$. Note that the method accesses $y_k$ only via the block-gradients $\nabla_i f(y_k)$ for $i \in S_k$. Hence, if it is possible to cheaply compute these gradients \emph{without} actually forming $y_k$, we can avoid full-dimensional operations.

We now show that this can be done for functions $f$ of the form \eqref{eq:shs5s9s}, where $f_j$ is as in Theorem~\ref{thm:compute_Lip}: \begin{equation}\label{eq:sjstiuy}
f(x)=\textstyle{\sum_{j=1}^m \phi_j(e_j^T A x)}.\end{equation}

Let $D_i$ be the set of such $j$ for which $A_{ji}\neq 0$. If we write $r_{\pp_k}=A \pp_k$ and $r_{\tilde{z}_k} = A\tilde{z}_k$, then using \eqref{eq:sjstiuy} we can  write
\begin{equation}\label{eq:nmasgsa89}
 \textstyle{ \nabla_i f(\theta_k^2 \pp_k+\tilde{z}_k) = \sum_{j \in D_i} A_{ji}^T \phi'_j(\theta_k^2 r_{\pp_k}^j + r_{\tilde{z}_k}^j ).}
\end{equation}

Assuming we store and maintain the residuals $r_{\pp_k}$ and $r_{\tilde{z}_k}$, the computation of the product $A_{ji}^T\phi_j'(\cdot)$ costs ${\cal O}(N_i)$ (we assume that the evaluation of the univariate derivative $\phi_j'$ takes ${\cal O}(1)$ time), and hence the computation of the  block derivative \eqref{eq:nmasgsa89} requires ${\cal O}(|D_i|N_i)$ arithmetic operations. Hence on average, computing all block gradients for $i \in S_k$ will cost
\[ \textstyle{C = \Exp\left[ \sum_{i\in \hat{S}} {\cal{O}}(|D_i|N_i) \right] = 
\frac{\tau}{n}\sum_{i=1}^n {\cal O}(|D_i|N_i).}\]

This will be small if $|D_i|$ are small and $\tau$ is small. For simplicity, assume all blocks are of equal size, $N_i=b = N/n$. Then
\[ \textstyle{C =  \frac{b\tau}{n} \times {\cal O}\left(\sum_{i=1}^n|D_i|\right) = \frac{b\tau}{n} \times {\cal O}\left(\sum_{j=1}^m\omega_j\right) = \frac{b\tau m }{n}{\cal O}(\bar{\omega}) =\tau \times {\cal O}\left(\frac{b  m \bar{\omega}}{n}\right).}\]

It can be easily shown that the maintenance of the residual vectors $r_{\pp_k}$ and $r_{\tilde{z}_k}$ takes the same amount of time ($C$) and hence the total work per iteration is $C$. In many practical situations, $m \leq n$, and often $m\ll n$ (we focus on this case in the paper since usually this corresponds to $f$ not being strongly convex) and $\bar{\omega}=O(1)$. This then means that $C=\tau \times {\cal O}(b)$. That is, each of the $\tau$ processors do work proportional to the size of a single block per iteration.

The favorable situation described above is the consequence of the block sparsity of the data matrix $A$ and does not depend on $\phi_j$ insofar as the evaluation of its derivative takes ${\cal O}(1)$ work. Hence, it applies to convex quadratics ($\phi_j(s)=s^2$), logistic regression ($\phi_j(r)=\log(1+\exp(s))$) and also to the smooth approximation $f_\mu(x)$ of $f(x)=\|Ax-b\|_1$, defined by
\[
 f_{\mu}(x)=\sum_{j=1}^m \norm{e_j^T A}_{w^*}^* \psi_\mu \left( \frac{\abs{e_j^T A x - b_j}}{ \norm{e_j^T A}_v^*} \right), \qquad \psi	_\mu(t)=\begin{cases}  \frac{t^2}{2\mu}, & 0 \leq t \leq \mu, \\ t - \frac{\mu}{2}, & \mu \leq t, \end{cases}
\]
with smoothing parameter $\mu>0$, as considered in \cite{Nesterov05:smooth, FR:spcdm}. Vector $w^*$ is as defined in \cite{FR:spcdm}; $\|\cdot\|_v$ is a weighted norm in $\R^m$.


\section{Numerical experiments}
\label{SEC:NUMERICS}

In all tests we used a shared-memory workstation with 32 Intel Xeon processors at 2.6~GHz and 128~GB RAM. In the experiments, we have departed from the theory in two ways: i) our implementation of APPROX is \emph{asynchronous} in order to limit communication costs, and ii) we approximated the $\tau$-nice sampling by a $\tau$-independent sampling as in~\cite{RT:PCDM} (the latter is very easy to generate in parallel; please note that our analysis can be very easily extended to cover the $\tau$-independent sampling). For simplicity, in all tests we assume all blocks are of size 1 ($N_i=1$ for all $i$). However, further speedups can be obtained by working with larger block sizes as then each processor is better utilized.


\subsection{The effect of new stepsizes}

In this experiment, we compare the performance of the new stepsizes ( introduced in Section~\ref{sec:neweso}) with those proposed  in~\cite{RT:PCDM} (see Table~\ref{tbl:more_stepsizes}).
We generated random instances of the $L_1$-regularized least squares problem (LASSO), 
\[f(x)=\frac{1}{2}\|Ax-b\|^2, \qquad \Reg(x) = \lambda \|x\|_1,\]
with various distributions of the separability degrees $\omega_j$ (= number of nonzero elements on the $j$th row of $A$) and studied the weighted distance to the optimum $ \norm{x_*-x_0}_{\vv}$ for the initial point $x_0=0$. This quantity appears in the complexity estimate \eqref{eq:k_0000} and depends on $\tau$ (the number of processors). We chose a random matrix of small size: $N=m=1000$ as this is sufficient to make our point, and consider $\tau\in \{10,100,1000\}$.

In particular, we consider three different  distributions of $\{\omega_j\}$: uniform, intermediate and extreme. The results are summarized in Table~\ref{tab:omegaunif}. 
First, we generated a \emph{uniformly} sparse matrix with $\omega_j=30$ for all $j$. In this case, $\vv^{\text{fr}}=\vv^{\text{rt}}$, and hence the results are the same.  We then generated an \emph{intermediate} instance, with $\omega_j = 1+\floor{30 j^2/m^2}$. The matrix has many rows with a few nonzero elements
and some rows with up to 30 nonzero elements. Looking at the table, clearly, the new stepsizes are better. The improvement is moderate when there are a few processors, but for $\tau=1000$, the complexity is 25\% better.
Finally, we generated a rather \emph{extreme} matrix with $\omega_1=500$ and $\omega_j=3$ for $j>1$. 
We can see  that the new stepsizes are much better, even with few processors, and can lead to $5\times$ speedup.

\begin{table}[h]
\centering
\begin{tabular}{|r|r|r|r|r|r|r|}
\hline
& \multicolumn{2}{ |c| }{Uniform} & \multicolumn{2}{ |c| }{Intermediate} & \multicolumn{2}{ |c| }{Extreme}\\
$\tau$ & $\norm{x^*}_{\vv^{\text{fr}}}$  & $ \norm{x^*}_{\vv^{\text{rt}}}$ & $\norm{x^*}_{\vv^{\text{fr}}}$  & $ \norm{x^*}_{\vv^{\text{rt}}}$ & $\norm{x^*}_{\vv^{\text{fr}}}$  & $ \norm{x^*}_{\vv^{\text{rt}}}$\\
\hline
10 &  10.82 & 10.82 &   6.12 & 6.43 & 2.78 & 5.43\\
\hline
100 & 19.00 & 19.00 & 9.30 & 11.38 & 4.31 & 16.08\\
\hline
1000 & 52.49 & 52.49 & 24.00 & 31.78 & 11.32 & 50.52\\
\hline
\end{tabular}
\caption{\footnotesize Comparison of ESOs in the uniform case}
\label{tab:omegaunif}
\end{table}

In the experiments above, we have first fixed a sparsity pattern and then generated a \emph{random} matrix $A$ based on it. However, much larger differences can be seen  for special matrices $A$. We shall now comment on this.

Consider the case $\tau=n$. In view of \eqref{eq:sks9s6s5}, the complexity of APPROX is proportional to $\|\vv\|_1$. 
Fix $\omega$ and $\omega_1,\dots,\omega_j$ and let us ask the question: for what data matrix $A$ will the ratio $\theta = \|\vv^{\text{rt}}\|_1/\|\vv^{\text{fr}}\|_1$ be maximized? Since $\|\vv^{\text{rt}}\|_1 = \omega \sum_j \|A_{j:}\|^2$ and $\|\vv^{\text{fr}}\|_1 = \sum_{j} \omega_j \|A_{j:}\|^2$, we the maximal ratio is given by

\[\max_{A} \theta \eqdef \max_{\alpha \geq 0} \left\{\omega \sum_{j=1}^m \alpha_j \;:\; \sum_{j=1}^m \omega_j \alpha_j\leq 1\right\} = \max_{j} \frac{\omega}{\omega_j}.\]

The extreme case is attained for some matrix with at least one  dense row ($\omega_j$) and one maximally sparse row ($\omega_j=1$), leading to $\theta=n$. So, there are instances for which the new stesizes can lead to an up to $n\times$ speedup for APPROX when compared to the stepsizes $\vv^{\text{rt}}$. Needless to say, these extreme instances are artificially constructed. 

\subsection{L1-regularized L1 regression}

We consider the data given in the dorothea dataset~\cite{platt199912}. It is a sparse moderate-sized feature matrix $A$ with $m$=800, $N$=100,000, $\omega$=6,061 and a vector $b \in \R^m$. We wish to find $x \in \mathbb{R}^N$ that  minimizes \[\norm{A x-b}_1+\lambda \norm{x}_1\] with $\lambda=1$.
Because the objective is nonsmooth and non-separable, we apply the smoothing technique presented in~\cite{Nesterov05:smooth} for the first part of the objective and use the smoothed parallel coordinate descent method proposed in \cite{FR:spcdm} (this methods needs special stepsizes which are studied in that paper). The level of smoothing depends on the expected accuracy: we chose $\epsilon=0.1$, which corresponds to 0.0125\% of the initial value. 

We compared 4 algorithms (see  Figure~\ref{fig:l1regul-l1regress}), all run with 4 processors.
As one can see, the coordinate descent method is very efficient on this problem. 
However, the accelerated coordinate descent is still able to outperform it. As the problem is of small size (which is sufficient for the sake of comparison), we could compute  the optimal solution using an interior point method for linear programming and compare the value at each iteration to the optimal value (Table~\ref{tab:l1reg}).
Each line of the table gives the time needed by APPROX and PCDM to reach a given accuracy target. In the beginning (until $F(x_k)-F(x^*)<6.4$), the algorithms 
are in a transitional phase. Then, when one runs the algorithm twice as long, $F(x_k)-F(x^*)$ is divided by 2 for SPCDM
and by 4 for APPROX. This highlights the difference in the convergence speeds: $O(1/k)$ compared to $O(1/k^2)$.
As a result, APPROX gives an $\epsilon$-solution in 156.5 seconds while SPCDM
has not finished yet after 2000 seconds. 

\begin{figure}
\centering
\includegraphics[width=22em]{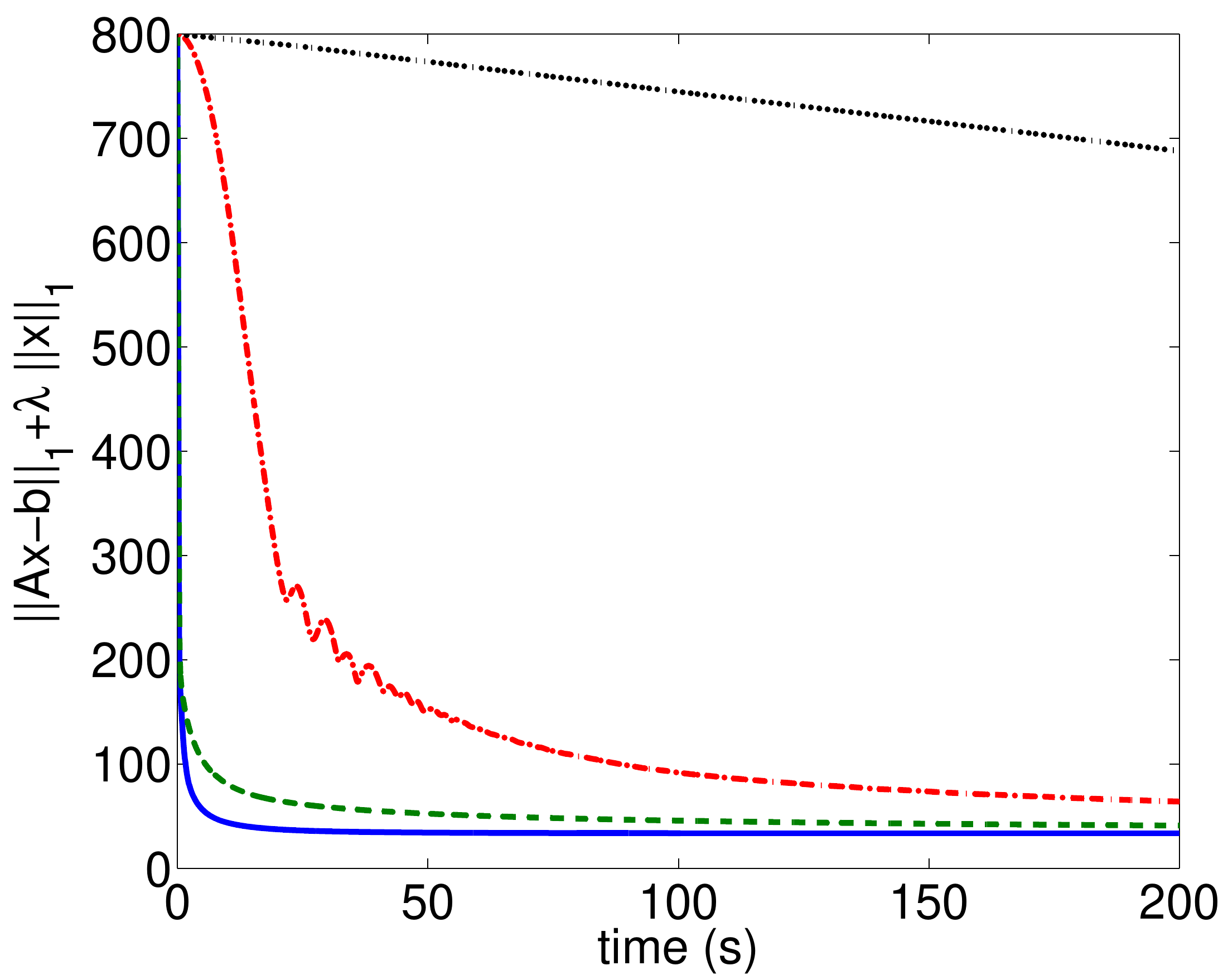}
\caption{\footnotesize Comparison of four algorithms for  $L_1$ regularized $L_1$ regression on the \texttt{dorothea} dataset:
gradient method (dotted black line), accelerated gradient method (\cite{Nesterov05:smooth}, dash-dotted red line), smoothed parallel  coordinate descent method (SPCDM~\cite{FR:spcdm}, dashed green line) and 
APPROX with stepsizes $\vv^{\text{fr}}$ (solid blue line). 
}
\label{fig:l1regul-l1regress}
\end{figure}

\begin{table}
\centering
\begin{tabular}{|r|r|r|}
\hline 
$F(x_k)-F(x_*)$ &  APPROX  & SPCDM \\
\hline
 409.6  & 0.2 s &  0.2 s \\
 204.8 & 0.3 s &  0.4 s \\
 102.4 & 1.0 s &    2.3 s \\
  51.2 & 2.2 s &  8.8 s \\ 
  25.6 & 4.5 s & 29.2 s \\
 12.8 & 8.3 s & 93.4 s \\  
6.4 & 14.4 s & 246.6 s \\
3.2 & 22.8 s & 562.3 s \\ 
 1.6 &  34.4  s &  1082.1 s  \\
   0.8 & 50.1 s &   1895.3 s\\
 0.4 & 71.8 s &  $>$2000 s\\
   0.2 & 103.4 s &  $>$2000 s\\
   0.1 &  156.5 s &   $>$2000 s\\         
\hline              
\end{tabular}
\caption{\footnotesize Comparison of objective decreases for APPROX and smoothed parallel coordinate descent (SPCDM) on a problem with $F(x)=\norm{A x-b}_1+\lambda \norm{x}_1$.}
\label{tab:l1reg}
\end{table}

\subsection{Lasso}

We now consider $L_1$ regularized least squares regression on the \texttt{KDDB} dataset~\cite{platt199912}. 
It consists of a medium size sparse feature matrix $A$ with $m=29,890,095$, $N=19,264,097$ and $\omega=75$, and a vector $b \in \R^m$. We wish to find $x \in \R^N$ that  minimizes \[F(x)=\frac{1}{2}\norm{A x-b}^2+\lambda \norm{x}_1\] with $\lambda=1$.

We compare APPROX (Algorithm~\ref{algo:w-simple}) with the (non-accelerated)
parallel coordinate descent method (PCDM~\cite{RT:PCDM}) in Figure~\ref{fig:lasso_kddb}, both run with $\tau=16$ processors.

Both algorithms converge quickly. PCDM is faster in the beginning because each iteration is half as expensive. However, APPROX is faster afterwards. For this problem, the optimal value is not known so it is  difficult to compare the actual accuracy.

Let us remark that an important feature of the $L_1$-regularization is that it promotes sparsity in the optimization variable $x$.
As APPROX only involves proximal steps on the $z$ variable, only $z_k$ is encouraged to be  sparse but not $x_k$, $y_k$ or $\pp_k$. A possible way to obtain a  sparse solution with APPROX is to first compute $x_k$ and then post-process with a few iterations of a sparsity-oriented method (such as iterative hard thresholding, full proximal gradient descent or  cyclic/randomized coordinate descent).

\begin{figure}
\centering
\includegraphics[width=22em]{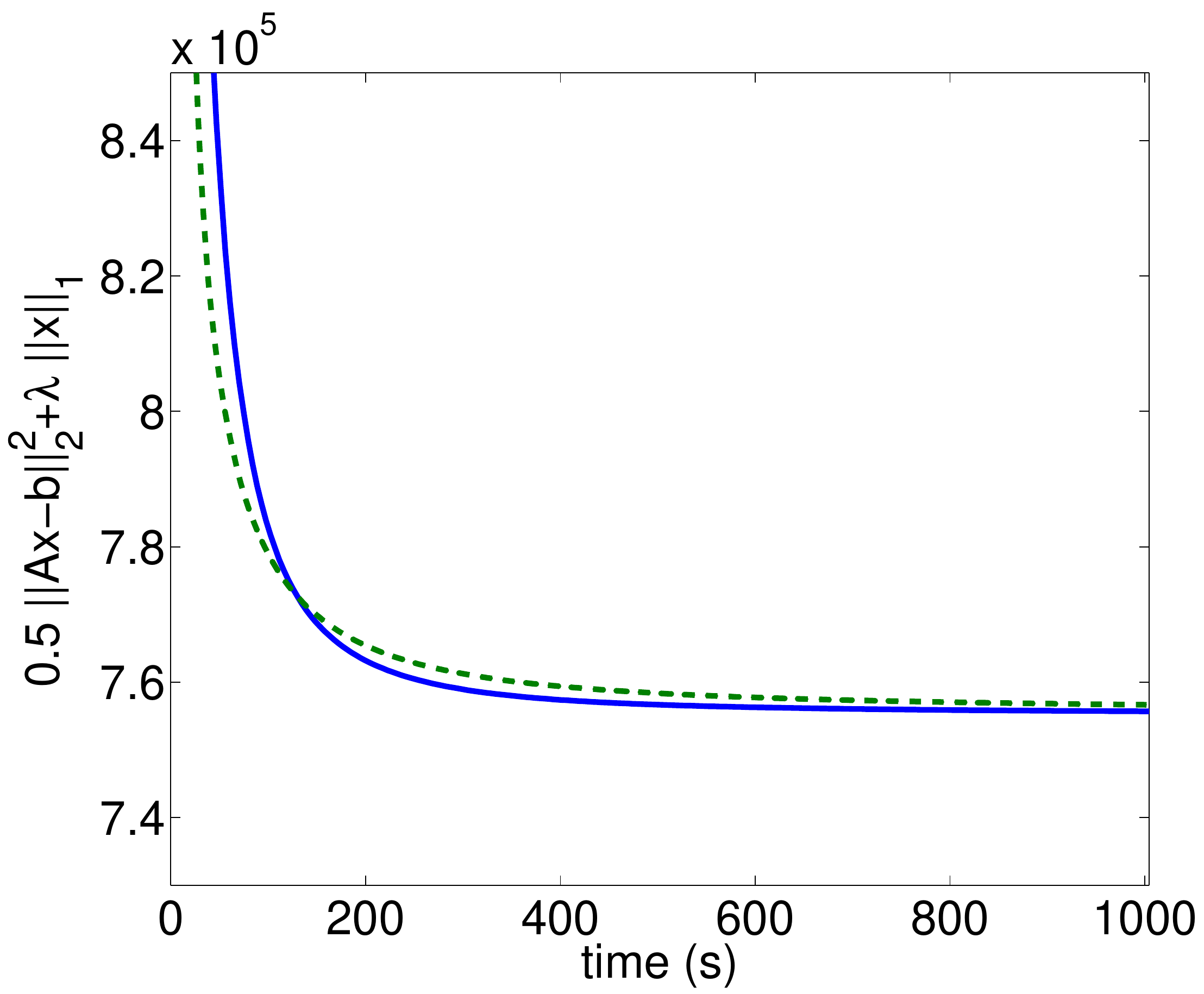}
\caption{\footnotesize Comparison of PCDM and APPROX for $l_1$ regularized least squares  on the kddb dataset.
As the decrease is very big in the first seconds (from 8.3 $10^8$ to 8.5 $10^5$), we present a zoom for $7.3 \leq F(x) \leq 8.5$.
Randomized coordinate descent~\cite{RT:PCDM}: dashed green line.
Accelerated coordinate descent (Algorithm~\ref{algo:w-simple}): solid blue line. 
}
\label{fig:lasso_kddb}
\end{figure}

\subsection{Training linear support vector machines}

Our last experiment is the dual of Support Vector Machine problem~\cite{ShalevTewari09}. For the dual SVM, the coordinates correspond to examples. 

We use the Malicious URL dataset~\cite{platt199912} with data matrix $A$ of size $m=2,396,130$, $N=3,231,961$ and a vector $b \in \mathbb{R}^N$. 
Here $\omega=n$ (and hence the data set is not particularly suited for parallel coordinate descent methods) but the matrix is still sparse (nnz=$277,058,644 \ll mn$).

We wish to find $x \in [0,1]^N$ that  minimizes 
\[F(x)=\frac{1}{2 \lambda N^2}\sum_{j=1}^m \left(\sum_{i=1}^N b_i A_{ji}x_i\right)^2 - \frac{1}{N} \sum_{i=1}^N x_i + I_{[0,1]^N}(x),\] with $\lambda=1/N$.
We compare APPROX (Algorithm~\ref{algo:w-simple}) with  Stochastic Dual Coordinate Ascent (SDCA~\cite{ShalevTewari09,takac2013mini}); the results are in Figure~\ref{fig:svm_url}. We have used a single processor only ($\tau=1$).

For this problem, one can recover a primal solution~\cite{ShalevTewari09} and thus we can  compare the decrease in the duality gap; summarized in Table~\ref{tab:dualsvm}. One can see that APPROX is about twice as fast as SDCA on this instance.

\begin{figure}
\centering
\includegraphics[width=22em]{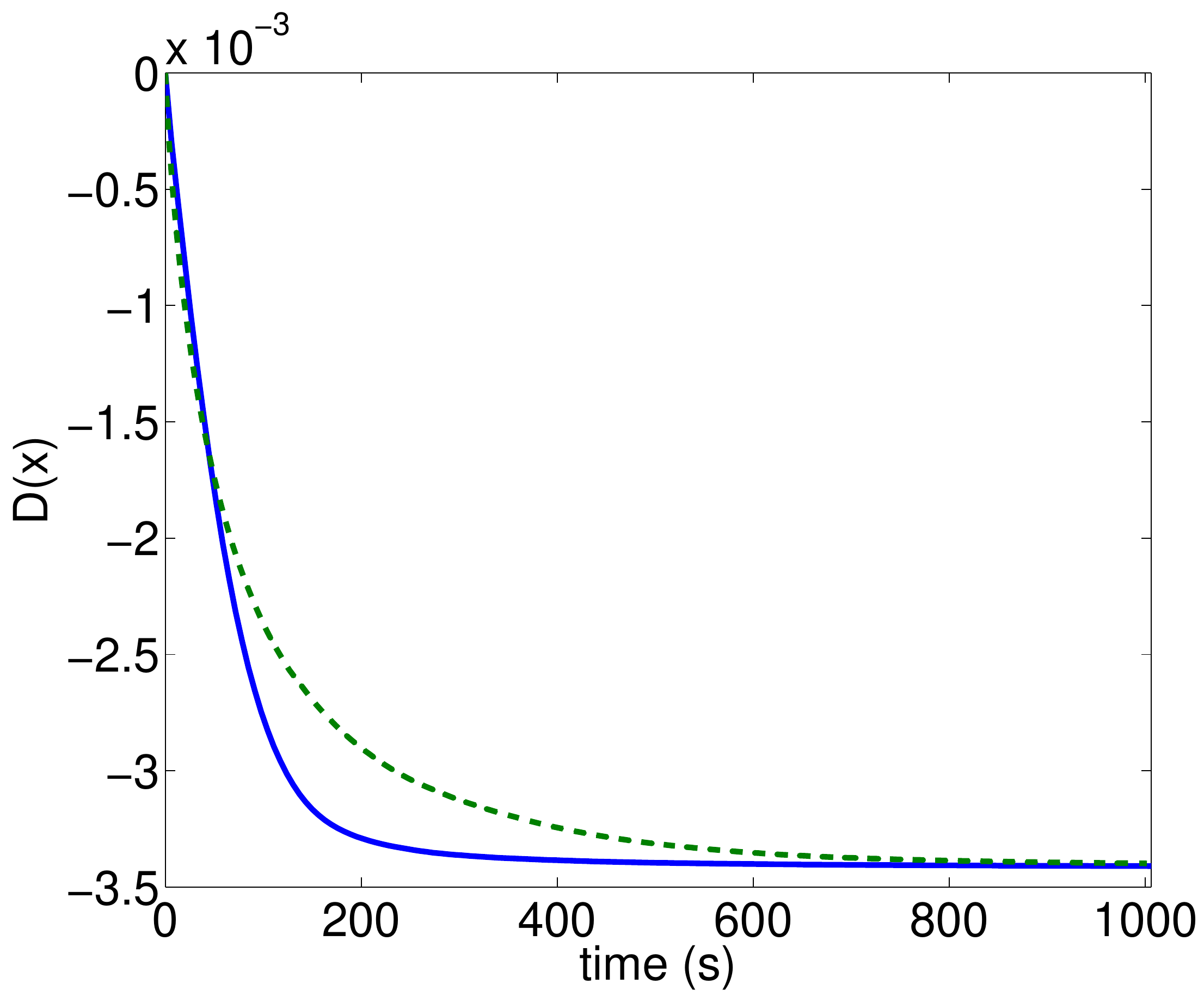}
\caption{\footnotesize Comparison of PCDM and APPROX for the dual of the Support Vector Machine problem on the Malicious URL dataset.
Randomized coordinate descent~\cite{RT:PCDM}: dashed green line.
Accelerated coordinate descent (Algorithm~\ref{algo:w-simple}): solid blue line. 
}
\label{fig:svm_url}
\end{figure}

\begin{table}
\centering
\begin{tabular}{|r|r|r|}
\hline 
Duality gap &  APPROX  & SDCA \\
\hline
 0.0256 &  33 s   &  26    s \\
  0.0128 &   59 s &   97  s \\
 0.0064 & 91 s &  206 s \\
 0.0032   &  137 s   & 310 s \\
 0.0016 & 182 s & 452 s \\
  0.0008    &   273 s  &  606  s \\
 0.0004  &   407 s& 864  s \\
  0.0002 &  614 s &  1148 s \\
 0.0001  &   954 s & 1712 s \\
\hline              
\end{tabular}
\caption{\footnotesize Decrease of the duality gap for accelerated parallel coordinate descent (APPROX) and stochastic dual coordinate ascent (SDCA).}
\label{tab:dualsvm}
\end{table}

\newpage
\section{Conclusion}

In summary, we  proposed APPROX: a  stochastic coordinate descent method combining the following \emph{four acceleration strategies:}

\begin{enumerate}
\item Our method is \emph{accelerated}, i.e., it achieves a $O(1/k^2)$ convergence rate. Hence, the method is better able to obtain a high-accuracy solution on non-strongly convex problem instances.
\item Our method is \emph{parallel}. Hence, it is able to better utilize modern parallel computing architectures and effectively taming the problem dimension $n$. 
\item We have proposed new \emph{longer  stepsizes} for faster convergence on functions whose degree of separability $\omega$ is larger than their  degree of separability $\bar{\omega}$.
\item We have shown that our method can be implemented \emph{without the need to perform full-dimensional vector operations.}
\end{enumerate}

\bibliographystyle{plain}
\bibliography{literature}

\appendix

\section{Proof of Proposition~\ref{prop:equivalence2} (equivalence)}
\label{SEC:EQUIVALENCE}

It is straightforward to see that  $x_0=y_0=z_0=\tilde{x}_0=\tilde{y}_0=\tilde{z}_0$ and hence the statement holds for $k=0$. By induction, assume it holds for some $k$.  Note that for $i \notin S_k$, $\tilde{z}_{k+1}^{(i)} = \tilde{z}_{k}^{(i)} = z_k^{(i)} = z_{k+1}^{(i)}$. If $i \in S_k$, then

\begin{equation}\label{eq:sjssjs8}\tilde{z}_{k+1}^{(i)} = \tilde{z}_k^{(i)} + t_k^{(i)},\end{equation}
where \begin{eqnarray}t_k^{(i)} &=& \arg\min_{t \in \R^{N_i}} \left\{ \langle \nabla_i f (\theta_k^2 \pp_k + \tilde{z}_k), t \rangle + \frac{n\theta_k \vv_i}{2\tau}  \|t\|_{(i)}^2 + \Reg_i(\tilde{z}_{k}^{(i)}+t)\right\}\notag\\
&\overset{\eqref{eq:yyyy}}{=}&\arg\min_{t \in \R^{N_i}} \left\{ \langle \nabla_i f (\tilde{y}_k), t \rangle + \frac{n\theta_k  \vv_i}{2\tau}  \|t\|_{(i)}^2 + \Reg_i(\tilde{z}_{k}^{(i)}+t)\right\}\notag\\
&=&\arg\min_{t \in \R^{N_i}} \left\{ \langle \nabla_i f (y_k), t \rangle + \frac{n\theta_k \vv_i}{2\tau}  \|t\|_{(i)}^2 + \Reg_i(z_{k}^{(i)}+t)\right\}\notag\\
&=& -z_k^{(i)} + \arg\min_{z \in \R^{N_i}} \left\{ \langle \nabla_i f (y_k), z-z_k^{(i)} \rangle + \frac{n\theta_k \vv_i}{2\tau}  \|z-z_k^{(i)}\|_{(i)}^2 + \Reg_i(z)\right\}\notag\\
&=& -z_k^{(i)} + \arg\min_{z \in \R^{N_i}} \left\{ \langle \nabla_i f (y_k), z-y_k^{(i)} \rangle + \frac{n\theta_k \vv_i}{2\tau}  \|z-z_k^{(i)}\|_{(i)}^2 + \Reg_i(z)\right\}\notag\\
&=& -z_k^{(i)} + z_{k+1}^{(i)}.\label{eq:shs7hs8s9s}
\end{eqnarray}
Combining \eqref{eq:sjssjs8} with \eqref{eq:shs7hs8s9s},  we get $\tilde{z}_{k+1}^{(i)} = \tilde{z}_k^{(i)}  -z_k^{(i)} + z_{k+1}^{(i)} = z_{k+1}^{(i)}$. Further, combining the two cases($i \in S_k$ and $i \notin S_k$), we arrive at
 \begin{equation}\label{eq:sjs8sjs}\tilde{z}_{k+1} = z_{k+1}.\end{equation} 

Now looking at the steps of Algorithm~\ref{algo:w-simple}, we see that
\begin{equation}\label{eq:isjs85s4}
\pp_{k+1}- \pp_k = -\frac{1-\tfrac{n}{\tau}\theta_k}{\theta_k^2}(\tilde{z}_{k+1}-\tilde{z}_k),
\end{equation}
and can thus write
\begin{eqnarray}
\tilde{x}_{k+1} & \overset{\eqref{eq:xxx}}{=} & \theta_{k}^2 \pp_{k+1} + \tilde{z}_{k+1}\notag\\
&\overset{\eqref{eq:isjs85s4}}{=}& \theta_k^2 \left(\pp_k - \frac{1-\frac{n}{\tau}\theta_k}{\theta_k^2}(\tilde{z}_{k+1}-\tilde{z}_k)\right) + \tilde{z}_{k+1}\notag\\
&=& \theta_k^2 \pp_k + \tilde{z}_k + \frac{n}{\tau}\theta_k (\tilde{z}_{k+1}-\tilde{z}_k)\notag\\
&\overset{\eqref{eq:yyyy}}{=}& \tilde{y}_k + \frac{n}{\tau}\theta_k (\tilde{z}_{k+1}-\tilde{z}_k)\notag\\
&\overset{\eqref{eq:sjs8sjs}}{=}&y_k + \frac{n}{\tau}\theta_k (z_{k+1}-z_k)\notag\\
&=& x_{k+1}.\label{eq:sjs8js}
\end{eqnarray}

Finally, 
\begin{eqnarray*}
\tilde{y}_{k+1} &\overset{\eqref{eq:yyyy}}{=}& \theta_{k+1}^2 \pp_{k+1}+ \tilde{z}_{k+1}\\
&\overset{\eqref{eq:xxx}}{=}& \frac{\theta_{k+1}^2 }{\theta_k^2}(\tilde{x}_{k+1}-\tilde{z}_{k+1}) + \tilde{z}_{k+1}\\
&\overset{\eqref{eq:theta_id}}{=}& (1-\theta_{k+1})(\tilde{x}_{k+1}-\tilde{z}_{k+1}) +\tilde{z}_{k+1}\\
&\overset{\eqref{eq:sjs8sjs}+\eqref{eq:sjs8js}}{=}& (1-\theta_{k+1})(x_{k+1}-z_{k+1}) + z_{k+1}\\
&=& y_{k+1},
\end{eqnarray*}
which concludes the proof.

\end{document}